\documentclass[reqno,11pt]{amsart}
\usepackage{amsmath,amssymb,mathrsfs,amsthm,amsfonts,comment}
\usepackage[inline]{enumitem}
\usepackage[usenames,dvipsnames]{xcolor}
\usepackage{hyperref}
\hypersetup{%
	colorlinks=true, linkcolor=blue,
	citecolor=ForestGreen
}
\usepackage{bbm}
\usepackage[paper=letterpaper,margin=1in]{geometry}

\newtheorem{theorem}{Theorem}[section]
\newtheorem{conjecture}[theorem]{Conjecture}
\newtheorem{lemma}[theorem]{Lemma}
\newtheorem{proposition}[theorem]{Proposition}
\newtheorem{corollary}[theorem]{Corollary}
\theoremstyle{definition}
\newtheorem{definition}[theorem]{Definition}
\newtheorem{remark}[theorem]{Remark}
\newtheorem{innercustomprop}{Proposition}
\newenvironment{customprop}[1]
  {\renewcommand\theinnercustomprop{#1}\innercustomprop}%
  {\endinnercustomprop}

\numberwithin{equation}{section}
\allowdisplaybreaks
\usepackage{acronym}

\renewcommand{\Pr}{\mathbb{P}}	
\newcommand{\Ex}{\mathbb{E}}	
\renewcommand{\d}{d}	
\newcommand{\ind}{\mathbf{1}}	

\newcommand{\norm}[1]{\Vert#1\Vert}
\newcommand{\cdrp}{\mathsf{CDRP}}
\newcommand{\m}{\mathsf}

\newcommand{\Con}{\mathrm{C}} 

\newcommand{\N}{\mathbb{N}}
\newcommand{\R}{\mathbb{R}} 
\newcommand{\Z}{\mathbb{Z}} 

\newcommand{\e}{\varepsilon}
\newcommand{\g}{\mathfrak{g}}
\newcommand{\h}{\mathfrak{h}}
\newcommand{\calH}{\mathcal{H}}
\newcommand{\calL}{\mathcal{L}}
\newcommand{\calZ}{\mathcal{Z}}
\renewcommand{\bar}{\overline}

\usepackage{graphicx}

\title[Tightness of the CDRP]{Short- and long-time path tightness of the continuum directed random polymer}

\author[S.\ Das]{Sayan Das}
\address{S.\ Das,
	Department of Mathematics, Columbia University,
	\newline\hphantom{\quad \ \ S. Das}
	2990 Broadway, New York, NY 10027 USA
}
\email{sayan.das@columbia.edu}

\author[W.\ Zhu]{Weitao Zhu}
\address{W.\ Zhu,
	Department of Mathematics, Columbia University,
	\newline\hphantom{\quad \ \ S. Das}
	2990 Broadway, New York, NY 10027 USA
}
\email{weitao.zhu@columbia.edu}


\begin{document}

\begin{abstract} We consider the point-to-point continuum directed random polymer ($\cdrp$) model that arises as a scaling limit from $1+1$ dimensional directed polymers in the intermediate disorder regime. We show that the annealed law of a point-to-point $\cdrp$ of length $t$ converges to the Brownian bridge under diffusive scaling when $t \downarrow 0$. In case that $t$ is large, we show that the transversal fluctuations of point-to-point $\cdrp$ are governed by the $2/3$ exponent. More precisely, as $t$ tends to infinity, we prove tightness of the annealed path measures of point-to-point $\cdrp$ of length $t$ upon scaling the length by $t$ and fluctuations of paths by $t^{2/3}$. The $2/3$ exponent is tight such that the one-point distribution of the rescaled paths converges to the geodesics of the directed landscape. This point-wise convergence can be enhanced to process-level modulo a conjecture. Our short and long-time tightness results also extend to point-to-line $\cdrp$. In the course of proving our main results, we establish quantitative versions of quenched modulus of continuity estimates for long-time $\cdrp$ which are of independent interest.
\end{abstract}

\maketitle

\section{Introduction}

\subsection{Background and motivation}\label{sec:bac}

Directed polymers in random environment can be considered as random walks interacting with a random external environment. First introduced and studied in \cite{huse}, \cite{imb} and \cite{bol}, they have since become a fertile ground for research in orthogonal polynomials, random matrices, stochastic PDEs, and integrable systems (see \cite{comets,gia,batesch} and the references therein). In the $(1+1)$-dimensional discrete polymer case, the random environment is specified by a collection of zero-mean i.i.d.~random variables $\{\omega=\omega(i,j) \mid (i,j)\in \Z_{+}\times \Z\}$. Given the environment, the energy of the $n$-step nearest neighbour random walk $(S_i)_{i=0}^n$ starting at the origin is given by 
$H_n^{\omega}(S):=\sum_{i=1}^n \omega(i,S_i).$
The \textbf{point-to-line} polymer measure on the set of all such paths is then defined as 
\begin{align*}
	\Pr_{n,\beta}^{\omega}(S)=\frac1{Z_{n,\beta}^{\omega}} e^{\beta H_n^{\omega}(S)} \Pr(S),
\end{align*}
where $\Pr(S)$ is the simple random walk measure, $\beta$ is the inverse temperature, and $Z_{n,\beta}^{\omega}$ is the partition function. 

A competition exists between the \textit{entropy} of paths and the \textit{energy} of the environment in this polymer measure. Spurred by this competition, two distinct regimes appear depending on the inverse temperature $\beta$. When $\beta = 0$ the polymer measure is the simple random walk; hence it is entropy-dominated and exhibits diffusive behavior. We refer to this scenario as \textit{weak disorder}. For $\beta>0$, the polymer measure concentrates on paths with high energies and the diffusive behavior ceases to be guaranteed. This type of energy domination is known as \textit{strong disorder}. For the definitions and results on the precise separation between the two regimes as well as results on higher dimensions, we refer the readers to \cite{comy,lacoin,cv}.

While the polymer behavior is characterized by diffusivity in weak disorder, the fluctuations of polymers in strong disorder are conjecturally characterized by two scaling exponents $\zeta$ and $\chi$ (\cite{timo}, \cite{akq2}):
\begin{align}\label{e:epoint}
	& \text{Fluctuation of the endpoint of the path: }|S_n|\sim n^{\zeta}, \\ & \text{Fluctuation of the log partition function: } [\log Z_{n,\beta}^{\omega} -\rho(\beta)n] \sim n^{\chi}. \notag
\end{align}
It is believed that directed polymers fall under the ``Kardar-Parisi-Zhang (KPZ) universality class" (see \cite{huse,hhf,kpz,ks2,corwin2012kardar}) with fluctuation exponent $\chi=\frac13$ and transversal exponent $\zeta=\frac23$.  This instance of the transversal exponent appearing larger than the diffusive scaling exponent $\frac12$ is called \textit{superdiffusivity}. 
Crucially, the conjectured values for $\chi$ and $\zeta$ satisfy the ``KPZ relation": 
\begin{align}\label{eq:kpzrel}
	\chi = 2\zeta -1.
\end{align}

At the moment, rigorous results on either exponent or the KPZ relation have been scarce. For directed polymers, $\zeta = 2/3$ has only been obtained for log-gamma polymers in \cite{timo,bcd21} and {for certain semi-discrete polymers called O'Connell-Yor polymer \cite{lan}}. Upper and lower bounds on $\zeta$ have been established in \cite{mp,om} under additional weight assumptions. For zero-temperature models, $\zeta = \frac{2}{3}$ has been established in \cite{joh,bcs,alan,dov,bghh}.
Outside the temperature models, the KPZ relation in \eqref{eq:kpzrel} has also been shown in other random growth models such as first passage percolation in \cite{chat} and \cite{ad} under the assumption that the exponents exist in a certain sense. In strong disorder, the polymer also exhibits certain localization phenomena (see \cite{comy,batesch,dz22} for partial surveys). In particular, the favorite region conjecture speculates that the endpoint of the polymer is asymptotically localized in a region of stochastically bounded diameter (see \cite{comets2013overlaps,batesch,bates,bak,dz22} for related results). 

Given the conceptual pictures on the two extreme regimes, in the present paper, we consider polymer fluctuations in the \textit{intermediate disorder regime}.  Introduced in \cite{akq2}, the intermediate disorder regime corresponds to scaling the inverse temperature $\beta = \beta_n = n^{-1/4}$ with the length of the polymer $n$, which captures the transitions between the weak and strong disorders and retains features of both. Within this regime, \cite{akq} showed that the partition function for point-to-point directed polymers has a universal scaling limit given by the solution of the Stochastic Heat Equation (SHE) for environment with finite exponential moments. In addition, the polymer path itself converges to a universal object called the Continuum Directed Random Polymer (denoted as $\cdrp$ hereafter) under the diffusive scaling.

We consider point-to-point $\cdrp$ of length $t$. The main contribution of this paper can be summarized as follows.
\begin{enumerate}[label=(\alph*), leftmargin=15pt]
	\item We show that as $t\downarrow 0$, the polymer paths behave diffusively and its annealed law converges in to the law of a Brownian bridge (Theorem \ref{thm:ann_short}).
	\item On the other hand, as $t\uparrow \infty$, the polymers have $t^{2/3}$ pathwise fluctuations. The latter result confirms superdiffusivity and the conjectural $2/3$ exponent for the $\cdrp$ (Theorem \ref{ltight} \ref{tght}). Moreover, the strength of our result exceeds the conjecture in \eqref{e:epoint}, which only claims endpoint tightness. Instead, in Theorem \ref{ltight} \ref{tght}, we prove that the annealed law of paths of point-to-point $\cdrp$ of length $t$ are tight (as $t\uparrow \infty$) upon $t^{2/3}$ scaling. This marks the first result of path tightness among all positive-temperature models.
	\item We also show pointwise weak convergence of the polymer paths under the $t^{2/3}$ scaling to points on the geodesic of the directed landscape (Theorem \ref{ltight} \ref{pt}). This ensures the $2/3$ scaling exponent is indeed tight. Modulo a conjecture on convergence of the KPZ sheet to the Airy Sheet (Conjecture \ref{conj:sheet}), we obtain pathwise convergence of the rescaled $\cdrp$ to the geodesic of the directed landscape (Theorem \ref{thm:ann_long_pr}). 
\end{enumerate}
These results provide a comprehensive picture of fluctuations of $\cdrp$ paths under short- and long-time scaling. Our short-time and long-time tightness results also extend to point-to-line $\cdrp$ (Theorem \ref{ltight.ptl}). The formal statement of the main results are given in Section \ref{sec:cdrp}.

\subsection{The model and the main results}\label{sec:cdrp}
We use the stochastic heat equation (SHE) with multiplicative noise to define the $\cdrp$ model. To start with, consider a four-parameter random field $\calZ(x,s;y,t)$ defined on 
\begin{align*}
	\R_{\uparrow}^4:= \{(x,s;y,t)\in \R^4 : s<t\}.
\end{align*}
For each $(x,s)\in \R\times \R$, the function $(y,t)\mapsto \calZ(x,s;y,t)$ is the solution of the SHE starting from location $x$ at time $s$, i.e., the unique solution of
\begin{align*}
	\partial_t\calZ=\tfrac12\partial_{xx}\calZ+\calZ\cdot \xi, \qquad (y,t) \in \R \times (s,\infty),
\end{align*}
with Dirac delta initial data $\lim_{t\downarrow s}\calZ(x,s;y,t)=\delta(x-y).$ Here $\xi=\xi(x,t)$ is the space-time white noise. The SHE itself enjoys a well-developed solution theory based on It\^o integral and chaos expansion \cite{bertini1995stochastic,walsh1986introduction} also \cite{corwin2012kardar,quastel2011introduction}. 
Via the Feynmann-Kac formula (\cite{hhf, comets}) the four-parameter random field can be written in terms of chaos expansion as
\begin{align}\label{eq:chaos}
	\calZ(x,s;y,t) = {p(y-x,t-s)}+\sum_{k=1}^{\infty}  \int_{\Delta_{k,s,t}} \int_{\R^k} \prod_{\ell=1}^{k+1} p(y_{\ell}-y_{\ell-1},s_{\ell}-s_{\ell-1}) \xi(y_{\ell},s_{\ell})  d\vec{y} \,d \vec{s},
\end{align}
with $\Delta_{k,s,t}:=\{(s_\ell)_{\ell=1}^k : s<s_1<\cdots<s_k<t\}$, $s_0=s, y_0=x, s_{k+1}=t$, and $y_{k+1}=y$. Here $$p(x,t):=(2\pi t)^{-1/2}\exp(-x^2/(2t))$$ denotes the standard heat kernel. The field $\calZ$ satisfies several other properties including the Chapman-Kolmogorov equations \cite[Theorem 3.1]{akq}. For all $0\le s<r<t$, and $x,y\in \R$ we have
\begin{align}\label{eq:chapman}
	\calZ(x,s;y,t)=\int_{\R} \calZ(x,s;z,r)\calZ(z,r;y,t) dz.
\end{align}
\begin{definition}[Point-to-point $\cdrp$]\label{def:cdrp} 
	Conditioned on the white noise $\xi$, let $\Pr^{\xi}$ be a measure on $C([s,t])$ whose finite-dimensional distribution is given by
	\begin{align}\label{eq:cdrp}
		\Pr^{\xi}(X({t_1})\in dx_1, \ldots, X({t_k})\in dx_k)=\frac1{\calZ(x,s;y,t)}\prod_{j=0}^k \calZ(x_j,t_j,;x_{j+1},t_{j+1})dx_1\cdots dx_k.
	\end{align}
	for $s=t_0\le t_1<\cdots<t_k\le t_{k+1}=t$, with $x_0=x$ and $x_{k+1}=y$. \eqref{eq:chapman} ensure $\Pr^{\xi}$ is a valid probability measure. Note that $\Pr^{\xi}$ also depends on $x$ and $y$ but we suppress it from our notations. We will use the notation $\cdrp(x,s;y,t)$ and write $X \sim \cdrp(x,s;y,t)$ when  $X(\cdot)$ is a random continuous function on $[s,t]$ with $X(s)=x$ and $X(t)=y$ and its finite-dimensional distributions given by \eqref{eq:cdrp} conditioned on $\xi$. We will also use the notation $\Pr^\xi, \Ex^\xi$ to denote the law and expectation conditioned on the noise $\xi$, and $\Pr, \Ex$ for the annealed law and expectation respectively.
\end{definition}
\begin{definition}[Point-to-line $\cdrp$]\label{def:cdrp2} 
	Conditioned on the white noise $\xi$, we let $\Pr_{*}^{\xi}$ be a measure $C([s,t])$ whose finite-dimensional distributions are given by
	\begin{align}\label{eq:cdrp2}
		\Pr_*^{\xi}(X({t_1})\in dx_1, \ldots, X({t_k})\in dx_k)=\frac1{\calZ(x,s;*,t)}\prod_{j=0}^k \calZ(x_j,t_j,;x_{j+1},t_{j+1})dx_1\cdots dx_k.
	\end{align}
	for $s=t_0\le t_1<\cdots<t_k\le t_{k+1}=t$, with $x_0=x$ and $x_{k+1}=*$. Here $\calZ(x,s;*,t):=\int_{\R} \calZ(x,s;y,t)dy$.
	Note that the Chapman-Kolmogorov equations \eqref{eq:chapman} ensure $\Pr_*^{\xi}$ is a probability measure. {The measure $\Pr_*^{\xi}$ also depends on $x$ but we again suppress it from our notations. We similarly use $\cdrp(x,y;*,t)$ to refer to random variables with $\Pr_*^{\xi}$ law.}
\end{definition}

\begin{remark}
	In both \cite{akq} and \cite{comets}, the authors considered a five-parameter random field  $\calZ_{\beta}(x,s;y,t)$ with inverse temperature $\beta$, which is the simultaneous solution of the stochastic heat equation
	\begin{align*}
		\partial_t\calZ_{\beta}=\tfrac12\partial_{xx}\calZ_{\beta}+\beta \calZ_\beta \xi, \qquad \lim_{t\downarrow s}\calZ_\beta(x,s;y,t)=\delta_x(y).
	\end{align*}
	and defined corresponding $\cdrp$ measures. Observe that
	when $\beta=0$, the stochastic heat equation becomes the heat equation and the corresponding $\cdrp$ measures reduce to Brownian measures. Furthermore, for any $\beta>0$, by the scaling property of the random field $\calZ_{\beta}$, i.e. (iii) of Theorem 3.1 in \cite{akq}, we have $$\calZ_{\beta}(x,s;y,t)\stackrel{d}{=} \beta^{-2}\calZ_1(\beta^2x,\beta^4s;\beta^2y,\beta^4t),$$
	Thus in this paper, we focus on exclusively on $\beta = 1$.
\end{remark}

We now state our first main result which discusses the annealed convergence of the $\cdrp$ in the short-time regime to Brownian bridge law.

\begin{theorem}[Annealed short-time convergence]\label{thm:ann_short} Fix $\e>0$. Let $X\sim \cdrp(0,0;0, \e)$. Consider the random function $Y^{(\e)}:[0,1]\to \R$ defined by $Y_{t}^{(\e)}:=\frac{1}{\sqrt{\e}}X({\e t})$. Let $\Pr^{\e}$ denote the annealed law of $Y^{(\e)}$ on the space of continuous functions on $C([0,1])$. As $\e \downarrow 0$, $\Pr^{\e}$ converges weakly to $\Pr_B$, where $\Pr_{B}$ is the measure on $C([0,1])$ generated by a Brownian bridge on $[0,1]$ starting and ending at $0$.
\end{theorem}

\begin{remark}  {The proof of Theorem \ref{thm:ann_short} appears in Section \ref{sec4.1}. With minor modification in the proof, the above theorem can be extended to include endpoints of the form $x\sqrt{\e}$. The resulting distributional limit is then a Brownian bridge on $[0,1]$ starting at $0$ and ending at $x$. We also remark that we expect Theorem \ref{thm:ann_short} to hold true even in the quenched case. However, some of our arguments, in particular the tightness, do not generalize to the quenched case. We hope to explore this direction in future works.}
\end{remark}

Our next result concerns the tightness and annealed convergence of the $\cdrp$ in the long-time regime and gives a rigorous justification of the $2/3$ scaling exponent discussed in Section \ref{sec:bac}. The limit is given in terms of the directed landscape constructed in \cite{dov,mqr} which arises as a universal full scaling limit of several zero-temperature models \cite{dv21}. Below we briefly introduce this limiting model before stating our result. 

The directed landscape $\calL$ is a random continuous function $\R_{\uparrow}^4 \to \R$ that satisfies the metric composition law
\begin{align}\label{def:metcomp}
	\calL(x,s;y,t)=\max_{z\in \R} \left[\calL(x,s;z,r)+\calL(z,r;y,t)\right],
\end{align} 
with the property that $\calL(\cdot,t_i;\cdot,t_i+s_i^3)$ are independent for any set of disjoint intervals $(t_i,t_i+s_i^3)$. As a function in $x,y$, $\calL(x,t;y,t+s^3)\stackrel{d}{=}s\cdot\mathcal{S}(x/s^2,y/s^2)$, where $\mathcal{S}(\cdot,\cdot)$ is a {parabolic} Airy Sheet. We omit definitions of the {parabolic} Airy Sheet (see Definition 1.2 in \cite{dov}) except that {$\mathcal{S}(0,\cdot)\stackrel{d}{=}\mathcal{A}(\cdot)$} where {$\mathcal{A}$ is the parabolic $\operatorname{Airy}_2$ process and $\mathcal{A}(x)+x^2$ is the (stationary) $\operatorname{Airy}_2$ process constructed in \cite{ps02}}.

\begin{definition}[Geodesics of the directed landscape] \label{def:geo}
	For $(x,s;y,t)\in \R_{\uparrow}^4$, a geodesic from $(x,s)$ to $(y,t)$ of the directed landscape is a random continuous function $\Gamma: [s,t]\to \R$ such that $\Gamma(s)=x$ and $\Gamma(t)=y$ and for any  $s\le r_1< r_2 < r_3 \le t$ we have
	\begin{align*}
		\calL\left(\Gamma(r_1),r_1;\Gamma(r_3),r_3\right)=\calL\left(\Gamma(r_1),r_1;\Gamma(r_2),r_2\right)+\calL\left(\Gamma(r_2),r_2;\Gamma(r_3),r_3\right).
	\end{align*}
	Thus the geodesic precisely contain the points where the equality holds in \eqref{def:metcomp}. Given any $(x,s;y,t)\in \R_{\uparrow}^4$, by Theorem 12.1 in \cite{dov}, it is known that almost surely there is a unique geodesic $\Gamma$ from $(x,s)$ to $(y,t)$.
\end{definition}

\begin{theorem}[Long-time $\cdrp$ path tightness]\label{ltight}
	Fix $\e > (0,1]$. $V \sim \cdrp(0, 0; 0, \e^{-1}).$  Define a random continuous function $L^{(\e)}:[0,1]\to \R$ as $L_t^{(\e)} := \e^{2/3}V({\e^{-1}t})$. We have the following:
	\begin{enumerate}[label=(\alph*), leftmargin=15pt]
		\item \label{tght} {Let $\Pr^{\e}$ denote the annealed law of $L^{(\e)}$, which is viewed as a random variable in the space of continuous functions on $[0,1]$ equipped with uniform topology and Borel $\sigma$-algebra. The sequence $\Pr^{\e}$ is tight w.r.t.~$\varepsilon$.} 
		\item \label{pt} For each $t\in (0,1)$, $L_t^{(\e)}$ converges weakly to $\Gamma(t\sqrt{2})$, where $\Gamma(\cdot)$ is the geodesic of directed landscape from $(0,0)$ to $(0,\sqrt{2})$. 
	\end{enumerate}
\end{theorem}
The above path tightness result under $2/3$ scaling is first such result among all positive-temperature models. Part \ref{pt} of the above theorem shows that this $2/3$ scaling is indeed correct: upon this scaling, the $\cdrp$ paths have pointwise non-trivial weak limit.

In the same spirit, we have the following short- and long-time tightness result for point-to-line $\cdrp$.

\begin{theorem}[Point-to-line $\cdrp$ path tightness]\label{ltight.ptl}
	Fix $\e \in (0,1]$. Suppose $X\sim \cdrp(0,0;*,\e)$ and $V \sim \cdrp(0, 0; *,\e^{-1}).$  Define two random continuous functions $Y_{*}^{(\e)}, L_{*}^{(\e)}:[0,1]\to \R$ as $Y_*^{(\e)}(t):=\e^{-1/2}X({\e t})$ and $L_*^{(\e)}(t) := \e^{2/3}V({\e^{-1}t})$. We have the following:
	\begin{enumerate}[label=(\alph*), leftmargin=15pt]
		\item \label{stght} {If we let $\Pr_{*,\texttt{s}}^{\e}$ denote the annealed law of $Y_*^{(\e)}(\cdot)$, then as $\varepsilon \downarrow 0$, $\Pr_{*,\texttt{s}}^{\e}$ converges weakly to $\Pr_{B_*}$, where $\Pr_{B_*}$ is the measure on $C([0,1])$ generated by a standard Brownian motion.}
		\item \label{ltght} If we let $\Pr_{*,\texttt{L}}^{\e}$ denote the annealed law of $L_*^{(\e)}(\cdot)$, then the sequence $\Pr_{*,\texttt{L}}^{\e}$ is tight w.r.t.~$\e$.
		\item \label{ptl} $L_*^{(\e)}(1)$ converges weakly to $2^{1/3}\mathcal{M}$, where $\mathcal{M}$ is the almost sure unique maximizer of $\operatorname{Airy}_2$ process minus {the parabola $x^2$}. 
	\end{enumerate}
\end{theorem}

We now explain how the pointwise weak convergence result in Theorem \ref{ltight} \ref{pt} can be upgraded to a process-level convergence modulo the following conjecture.

\begin{conjecture}[KPZ sheet to Airy sheet]\label{conj:sheet}
	Set $\h_{t}(x, y):= t^{-1/3}[\log\calZ(t^{2/3}x, 0; t^{2/3}y, t)+ \frac{t}{24}].$ As $t\to \infty$ we have the following convergence in law (as functions in $(x,y)$)
	\begin{align*}
		2^{1/3}\h_{t}(2^{1/3}x, 2^{1/3}y)\stackrel{d}{\rightarrow} \mathcal{S}(x, y)
	\end{align*}
	in the uniform-on-compact topology.  Here $\mathcal{S}$ is the parabolic Airy sheet. 
\end{conjecture}

When either $x$ or $y$ is fixed, the above weak convergence as a function in one variable is proven in \cite{qs20}. For zero-temperature models, such convergence has been shown recently in \cite{dv21} for a large class of integrable models. It remains to show that their methods can be extended to prove the Airy sheet convergence for positive-temperature models such as above. 

\medskip

Assuming the validity of Conjecture \ref{conj:sheet}, we can strengthen Theorem \ref{ltight} \ref{pt} to the following statement.
\begin{theorem}[Process annealed long-time convergence]\label{thm:ann_long_pr}
	Fix $\e>0$. Let $V\sim \cdrp(0, 0;0, \e^{-1})$. Define $L_{t}^{(\e)}:=\e^{2/3}V(\e^{-1}t)$, $t \in [0,1]$. This scaling produces a measure on $C([0,1])$ for each $\e>0$ conditioned on $\xi$. Assume Conjecture \ref{conj:sheet}. For $t\in (0,1),$ $\e \downarrow 0$, the annealed law of $L_{t}^{(\e)}$ as a process in $t$ converges weakly to $\Gamma(\sqrt{2}t)$, where $\Gamma(\cdot)$ is the geodesic of the directed landscape $\calL$ from $(0,0)$ to $(0, \sqrt{2})$.
\end{theorem}

\subsection{Proof Ideas}

Our main result on short-time and long-time tightness of $\cdrp$ (i.e., Theorems \ref{thm:ann_short}, \ref{ltight} and \ref{ltight.ptl}) follows a host of efforts that attempts to unravel the geometry of $\cdrp$ paths. In \cite{dz22}, the authors showed that the quenched density of point-to-point long-time $\cdrp$ exhibit pointwise localization. In particular, they showed any particular point on a point-to-point $\cdrp$ of length $t$ lives within a order $1$ window of a `favorite site' (depending only on the environment) and this favorite site varies in a $t^{2/3}$ window upon changing the environment. This suggests that the annealed law of polymers are within $t^{2/3}$ window \textit{pointwise}. Our theorems on long-time tightness extend this result to the \textit{full path} of the polymers. 

One of the key ingredients behind our tightness proofs is a detailed probabilistic understanding of the log-partition function of $\cdrp$. The log of the partition function of point-to-point $\cdrp$, i.e.,
\begin{align}
	\calH(x,s;y,t):=\log \calZ(x,s;y,t)
\end{align}
solves the KPZ equation with narrow wedge initial data. Introduced in \cite{kpz} as a model for random growth interfaces, KPZ equation has been extensively studied in both the mathematics and the physics communities (see \cite{ferrari2010random,quastel2011introduction,corwin2012kardar, hairer2013solving,hairer2014theory,quastel2015one,chandra2017stochastic,corwin2019} and the references therein). In \cite{acq}, the authors showed the one-point distribution of the KPZ equation $\calH(x,t):=\calH(0,0;x,t)$, has limiting Tracy-Widom GUE fluctuations of the order $t^{1/3}$ as $t\uparrow \infty$ (long-time regime), whereas fluctuations are  Gaussian of the order $t^{1/4}$ as $t\downarrow 0$ (short-time regime). Detailed information of the one-point tails of $\calH(x,t)$ as well as tail for the spatial process $\calH(\cdot,t)$ are rigorously proved in the mathematics works \cite{utail,ltail,cgh,t19,dt19} for long-time regime and in \cite{dg,lt21,llt21,t22} for short-time regime. 

For brevity, we only sketch the proof for our long-time path tightness result. The proof of short-time path tightness uses a relation of annealed law of $\cdrp$ with that of Brownian counterparts (Lemma \ref{l3}). The finite-dimensional convergence for the short-time case (Theorem \ref{thm:ann_short}) follows from chaos expansion and the same results for the long-time regime (Theorem \ref{ltight} \ref{pt} and Theorem \ref{ltight.ptl} \ref{ptl}) follow from the localization results in \cite{dz22}. Let us take a long-time polymer $V \sim \cdrp(0,0;0,\e^{-1})$ and scale it according to long-time scaling $L_t^{(\e)}=\e^{-2/3}V(\e^{-1}t)$ for $t\in [0,1]$. By the definition of the $\cdrp$ (Definition \ref{def:cdrp}), we see that the joint law of $(L_s^{(\e)},L_t^{(\e)})$ (where $0<s<t<1$) is proportional to $$\e^{-4/3}\exp\left[\Lambda_{(s,t);\e}(x,y)\right]$$
where
\begin{align}
	\Lambda_{(s,t);\e}(x,y):= \calH\big(0,0;x\e^{-\frac23},\tfrac{s}{\e}\big)+\calH\big(x\e^{-\frac23},\tfrac{s}{\e};y\e^{-\frac23},\tfrac{t}{\e}\big)+\calH\big(y\e^{-\frac23},\tfrac{t}{\e};0,\tfrac{1}{\e}\big)+\m{Err}_{(s,t);\e}. 
\end{align}
Here $\m{Err}_{(s,t);\e}$ is a correction term free of $x,y$ that one needs to add to extract meaningful fluctuation and tail results for the KPZ equation (see statement of Lemma \ref{l:ltail}). This correction term does not affect the joint density as it can be absorbed into the proportionality constant.

We next proceed to understand behaviors of the process $(x,y) \mapsto \Lambda_{(s,t);\e}(x,y)$. 
From \cite{acq}, it is known that for each fixed $s<t$ and $y\in \R$, the process $x\mapsto [\calH(x,s;y,t)+\frac{(x-y)^2}{2(t-s)}]$ is  stationary. Naively speaking, $x\mapsto \calH(x,s;y,t)$ looks like a negative parabola: $-\frac{(x-y)^2}{2(t-s)}$. Thus it is natural to expect 
\begin{align}
	\e^{1/3} \Lambda_{(s,t);\e}(x,y) \approx -\frac{x^2}{2s}-\frac{(y-x)^2}{2(t-s)}-\frac{y^2}{2(1-t)}.
\end{align}
One of the technical contributions of this paper is to rigorously prove the above approximation holds for all $x,y$. Given any $\nu>0$, we show with probability at least $1-\Con \exp(-\frac1\Con M^2)$,

$$\e^{1/3}\Lambda_{(s,t);\e}(x,y) \le M-(1-\nu)\left[\frac{x^2}{2s}+\frac{(y-x)^2}{2(t-s)}+\frac{y^2}{2(1-t)}\right], \mbox{ for all }x,y\in \R.$$

\smallskip

The precise statement of the above result appears in Lemma \ref{l:ltail}. This multivariate process estimate allows us to conclude the quenched density of $(L_s^{(\e)},L_t^{(\e)})$ at $(x,y)$ is exponentially small, whenever $\frac{|x-y|}{\sqrt{t-s}}\to \infty$. Armed with this understanding of quenched density, in Proposition \ref{p:ltd}, we show that given any $\delta>0$, with probability at least $1-\Con\exp(-\tfrac1\Con M^2)$ we have
\begin{align*}
	{|L_{s}^{(\e)}-L_t^{(\e)}|} \le M|t-s|^{\frac12-\delta}.
\end{align*}
In fact the sharp decay estimates of quenched density (Lemma \ref{l:ltail}) allows us to prove a quenched version of the above statement (Proposition \ref{p:ltd}). Due to exponentially tight probability bounds of the above two-point differences, Proposition \ref{p:ltd} can be extended to quenched modulus of continuity estimates (Proposition \ref{p:qmc}) by standard methods. This leads to the path tightness of long-time $\cdrp$.  

\subsection*{Outline}
The rest of the paper is organized as follows. Section \ref{sec:tools} reviews some of the existing results related to the KPZ equation before proving a useful result on the short-time local fluctuations of the KPZ equation (Proposition \ref{p:stlf}). We then prove in Section \ref{sec:moc} a multivariate spatial process tail bound (Lemma \ref{l:ltail}) and modulus of continuity results (Propositions \ref{p:ltd} and \ref{p:ltd.ptl}) that culminate in the quenched modulus of continuity estimate in Proposition \ref{p:qmc} and Proposition \ref{p:qmc.ptl}. In Section \ref{sec:annconv}, we prove Theorems \ref{thm:ann_short}, \ref{ltight}, and \ref{ltight.ptl}, and Theorem \ref{thm:ann_long_pr} (modulo Conjecture \ref{conj:sheet}). Lastly, proof of a technical lemma used in Section \ref{sec:tools} appears in Appendix \ref{app1}.

\subsection*{Acknowledgements} The authors thank Ivan Corwin, Shirshendu Ganguly and Promit Ghosal for suggesting the problem and useful discussions. The authors are grateful to the anonymous referee for their careful reading and many valuable comments on improving our manuscript. The project was initiated during the authors' participation in the `Universality and Integrability in Random Matrix Theory and Interacting Particle Systems' research program hosted by the Mathematical Sciences Research Institute (MSRI) in Berkeley, California in fall 2021. The authors thank the program organizers for their hospitality and acknowledge the support from NSF DMS-1928930.

\section{Short- and long-time tail results for KPZ equation} \label{sec:tools}

Throughout this paper we use $\Con = \Con(x,y,z, \ldots) > 0$ to denote a generic deterministic positive finite constant that may change from line to line, but dependent on the designated variables $x,y,z, \ldots$. We use sans serif fonts such as $\m{A}, \m{B}, \ldots$ to denote events and $\neg\m{A}, \neg\m{B}, \ldots$ to denote their complements.

\smallskip

In this section, we collect several estimates related to the short-time and long-time tails of the KPZ equation. We record existing estimates from the literature in Proposition \ref{p:ltkpzeq} and Proposition \ref{p:stkpzeq}. These estimates form crucial tools to our later proofs. For our analysis, we also require an estimate on the short-time local fluctuations of the KPZ equation which is not available in the literature.  We present this new estimate in Proposition \ref{p:stlf}. Its proof appears at the end of this section.

Recall the four-parameter stochastic heat equation $\calZ(x,s;y,t)$ from \eqref{eq:chaos}. We set
\begin{align}\label{calhd}
	\calH(x,s;y,t):=\log\calZ(x,s;y,t).
\end{align}
When $x=s=0$, we use the abbreviated notation $\calH(y,t):=\calH(0,0;y,t)$. As mentioned in the introduction, fluctuation and scaling of the KPZ equation varies as $t\downarrow 0$ (short-time) and $t\uparrow \infty$ (long-time). For the two separate regimes we consider the following scalings:
\begin{equation}
	\begin{aligned}\label{eq:kpzscal}
		\g_{s,t}(x,y) &: = \frac{\calH(\sqrt{\frac{\pi(t-s)}{4}}x, s;\sqrt{\frac{\pi(t-s)}{4}}y, t ) + \log\sqrt{2\pi (t-s)}}{(\frac{\pi(t-s)}{4})^{1/4}} & \mbox{  for the short-time regime,}\\
		\h_{s,t}(x,y) & := \frac{\calH((t-s)^{2/3}x, s; (t-s)^{2/3}y, t) + \frac{t-s}{24}}{(t-s)^{1/3}} & \mbox{ for the long-time regime.}
	\end{aligned}
\end{equation}
We will often refer to the above bivariate functions as short-time and long-time KPZ sheet. In particular, when both $s =0$ and $x = 0$, we use the shorthands
$\g_t(y):=\g_{0,t}(0,y)$, and $\h_t(y):=\h_{0,t}(0,y)$. 

\begin{remark}\label{re:gscal} The above scalings satisfy several distributional identities. For fixed $s<t$ and $y\in \R$, from chaos representation for SHE it follows that
	\begin{align*}
		\calZ(0,s;x,t)\stackrel{d}{=}\calZ(0,s;-x,t), \quad \calZ(x,s;y,t) \stackrel{d}{=} \calZ(0,0;y-x,t-s).
	\end{align*}
	where the equality in distribution holds as processes in $x$.  This leads to  $\g_{s,t}(x,y) \stackrel{d}{=} \g_{t-s}(x-y)$ and $\h_{s,t}(x,y) \stackrel{d}{=} \h_{t-s}(x-y),$ as processes in $x$. 
\end{remark}
The following proposition collects several probabilistic facts for the long-time rescaled KPZ equation.
\begin{proposition}\label{p:ltkpzeq} Recall $\h_t(x)$ from \eqref{eq:kpzscal}. The following results hold:
	\begin{enumerate}[label=(\alph*), leftmargin=15pt]
		\item \label{p:lstat} For each $t>0$, $\h_t(x)+{x^2}/{2}$ is stationary in $x$.
		\item \label{p:ltail} Fix $t_0>0$. There exists a constant $\Con=\Con(t_0)>0$ such that for all $t\ge t_0$ and $s>0$ we have
		\begin{align*}
			\Pr\left(|\h_t(0)|\ge s\right) \le \Con\exp\left(-\tfrac1\Con s^{3/2}\right).
		\end{align*}
		\item \label{p:ltlf} Fix $t_0 > 0$. There exists a constant $\Con = \Con(t_0) > 0$ such that for all $x \in \R, s> 0$, $t\ge t_0$, and $\gamma \in (0,1],$ we have
		\begin{align*}
			\Pr\left(\sup_{z \in [x, x+\gamma]}\left| \h_t(z)+ \tfrac{z^2}{2} - \h_t(x) - \tfrac{x^2}{2}\right| \ge s\sqrt{\gamma}\right) \le \Con\exp\left(-\tfrac1\Con s^{3/2}\right). 
		\end{align*}
	\end{enumerate}
\end{proposition}

The results in Proposition \ref{p:ltkpzeq} are a culmination of results from several papers. Part \ref{p:lstat} follows from \cite[Corollary 1.3 and Proposition 1.4]{acq}. The one-point tail estimates for KPZ equation are obtained in \cite{utail,ltail}. One can derive part \ref{p:ltail} from those results or can combine the statements of Proposition 2.11 and 2.12 in \cite{cgh} to get the same. Part \ref{p:ltlf} is Theorem 1.3 from \cite{cgh}.

The study of short-time tails was initiated in \cite{dg}. Below we recall some known results from the same paper.
\begin{proposition}\label{p:stkpzeq} Recall $\g_t(x)$ from \eqref{eq:kpzscal}. The following results hold:
	\begin{enumerate}[label=(\alph*), leftmargin=15pt]
		\item \label{p:sstat} For each $t>0$, $\g_t(x)+\tfrac{(\pi t/4)^{3/4}}{2t}x^2$ is stationary in $x$.
		\item \label{p:stail} There exists a constant $\Con>0$ such that for all $t\le 1$ and $s>0$ we have
		\begin{align*}
			\Pr(|\g_t(0)|> s) \le \Con\exp\left(-\tfrac1\Con s^{3/2}\right).
		\end{align*}
	\end{enumerate}
\end{proposition}
Part \ref{p:sstat} follows from \cite[Lemma 2.11]{dg}. The one-point tail estimates for short-time rescaled KPZ equation are obtained in \cite[Corollary 1.6, Theorem 1.7]{dg}, from which one can derive part \ref{p:stail}.

For convenience, we write $m_t(x):=(\frac{\pi t}{4})^{3/4}\frac{x^2}{2t}$ to denote the parabolic term associated to the short-time scaling. The following result concerns the short-time analogue of Proposition \ref{p:ltkpzeq} \ref{p:ltlf}.

\begin{proposition}[Short-time local fluctuations of the KPZ equation]\label{p:stlf}
	There exists a constant $\Con>0$ such that for all $t\in (0,1)$, $x\in \R$, $\gamma\in (0,\sqrt{t})$ and $s>0$ we have
	\begin{align}\label{p24}
		\Pr\left(\sup_{z \in [x, x+\gamma]}\left|\g_t(z)+ m_t(z) - \g_t(x) - m_t(x)\right|\ge  s\sqrt\gamma\right) \le \Con\exp\left(-\tfrac1{\Con} s^{3/2}\right). 
	\end{align}
\end{proposition}
\begin{remark}\label{remk}
	The parabolic term $m_t(x)$ is steeper (as $t\le 1$) than the usual parabola that appears in the long-time scaling. This is the reason why Proposition \ref{p:stlf} requires $\gamma<\sqrt{t}$, whereas Proposition \ref{p:ltkpzeq} \ref{p:ltlf} holds for all $\gamma\in (0,1]$. 
\end{remark}

The proof of Proposition \ref{p:stlf} follows the same strategy as those of Proposition 4.3 and Theorem 1.3 in \cite{cgh} which employ the Brownian Gibbs property of the KPZ line ensemble (see \cite{CH16}). The same Brownian Gibbs property continues to hold for short-time $\g_t(\cdot)$ process (see Lemma 2.5 (4) in \cite{dg}). We include the proof of Proposition \ref{p:stlf} below for completeness after first describing its key proof ingredient.

\medskip

{We recall a property of $\g_t(\cdot)$ under \textit{monotone} events. Given an interval $[a,b]$, {we denote $\mathcal{B}(C([a,b]))$ to be the Borel $\sigma$-algebra on $C([a,b])$ generated by the uniform norm topology.} We call an event $A \in \mathcal{B}(C([a,b]))$ monotone w.r.t.~$[a,b]$ if for every pair of functions $f,g\in [a,b]\to \R$ with $f(a)=g(a)$, $f(b)=g(b)$ and $f(x)\ge g(x)$ for all $x\in (a,b)$,  we have
	\begin{align}\label{tom}
		f(x)\in A \implies g(x) \in A.
	\end{align}
	We call $(\mathfrak{a},\mathfrak{b})$ a stopping domain for $\g_t(\cdot)$ if $\{\mathfrak{a}\le a,\mathfrak{b}\ge b\}$ is measurable w.r.t.~$\sigma$-algebra generated by $(\g_t(x))_{x\not\in (a,b)}$ for all $a, b \in \R.$ A crucial property is the following:}
{\begin{lemma}\label{l:mon} Fix any $t>0$. For any $[a,b]\subset \R$, and a monotone set $A \in \mathcal{B}(C([a,b]))$ (w.r.t.~$[a,b]$), we have
		\begin{align}\label{mon}
			\Pr\left[\g_t(\cdot)\mid_{[a,b]} \ \in A \mid (\g_t(x))_{x\not\in (a,b)}\right] \le \Pr_{\operatorname{free}}^{(a,b),(\g_t(a),\g_t(b))}(A)
		\end{align}
		where $\Pr_{\operatorname{free}}^{(a,b),(y,z)}$ denotes the law of Brownian bridge on $[a,b]$ starting at $y$ and ending at $z$. Furthermore \eqref{mon} continues to hold if $(a,b)$ is a stopping domain for $\g_t(\cdot)$.
	\end{lemma}
	We will abuse our definition and call $\{\g_t(\cdot)\mid_{[a,b]} \ \in A \}$ to be monotone w.r.t.~$[a,b]$ if $A$ is monotone w.r.t.~$[a,b]$. The proof of the above lemma follows by utilizing the notion of the KPZ line ensemble and its Brownian Gibbs property \cite{CH16,dg}. We defer its proof and the necessary background on the KPZ line ensemble to Appendix \ref{app1}. }

\begin{proof}[Proof of Proposition \ref{p:stlf}] {Assume $s\ge 100$. For $s\le 100$, the constant $\Con>0$ can be adjusted so that the proposition holds trivially. We fix a $t_0\in (0,1)$ such that for all $s\ge 100$, and $t\le t_0$ we have
\begin{align}
    \label{onte} \tfrac14s\ge t^{1/4}(s+m_t(2))=t^{1/4}s+2(\pi/4)^{3/4}.
\end{align}
Let us first consider $t\in [t_0,1]$. We use the scalings from \eqref{eq:kpzscal} to get
\begin{align}
    \g_t(x)+m_t(x)=\tfrac{1}{\sqrt{r_t}}\left(\h_t(r_tx)+\tfrac{r_t^2x^2}{2}\right)+c_t,
\end{align}
where $r_t:=t^{-1/6}\sqrt{\pi/4}$ and $c_t:=(\pi t/4)^{-1/4}(\sqrt{2\pi t}-t/24)$. Take any $x\in \R$ and $\gamma\in (0,\sqrt{t})$. We have $r_t\gamma \le 1$. Setting $y:=r_tx$ and then applying Proposition \ref{p:ltkpzeq} \ref{p:ltlf} with $x\mapsto y$ and $\gamma\mapsto r_t\gamma$ we get
\begin{align}
    \mbox{l.h.s.~of \eqref{p24}} = 	\Pr\left(\sup_{z \in [y, y+r_t\gamma]}\left|\h_t(z)+ \tfrac{z^2}{2} - \h_t(y) - \tfrac{y^2}{2}\right|\ge  s\sqrt{\gamma r_t}\right) \le \Con\exp\left(-\tfrac1{\Con}s^{3/2}\right).
\end{align}
Let us now assume $t\le t_0$.} By Proposition \ref{p:stkpzeq} \ref{p:sstat}, we know that the process $\g_t(x)+m_t(x)$ is stationary in $x$. Thus it suffices to prove Proposition \ref{p:stlf} with $x=0$. Consider the following events
	$$\m{G}_{\gamma, s}: = \bigcap_{x \in \{\gamma-2, 0, \gamma, 2\}}\left\{-\tfrac{s}{4} \le \g_t(x) + m_t(x) \le \tfrac{s}{4}\right\},$$ 
	$$\m{Fall}_{\gamma, s}: = \bigg\{\inf_{z \in [0, \gamma]}(\g_t(z)+m_t(z))\le \g_t(0) -s\gamma^{1/2}\bigg\},$$
	$$\m{Rise}_{\gamma,s}:=\bigg\{\sup_{z\in [0, \gamma]}(\g_t(z) +m_t(z))\ge \g_t(0) + s\gamma^{1/2}\bigg\}.$$ 
	By one-point tail bounds from Proposition \ref{p:stkpzeq} \ref{p:stail} we have that $\Pr(\neg \m{G}_{\gamma,s})\le \Con\exp(-\frac1\Con s^{3/2})$. Thus, to show the proposition, it suffices to verify the following two bounds: 
	\begin{align}\label{e:bounds2}
		\Pr\left(\m{Fall}_{\gamma, s}, \m{G}_{\gamma, s}\right)\le \Con\exp\left(-\tfrac1\Con s^{2}\right), \qquad  \Pr\left(\m{Rise}_{\gamma, s}, \m{G}_{\gamma, s}\right)\le \Con\exp\left(-\tfrac1\Con s^{2}\right).
	\end{align}
	We begin with the $\m{Fall}_{\gamma,s}$ bound in \eqref{e:bounds2}. Clearly $\m{Fall}_{\gamma,s}$ event is monotone w.r.t.~$[0,2]$, by Lemma \ref{l:mon} we have
	$$\Pr\left(\m{Fall}_{\gamma, s}\mid (\g_t(x))_{x\in (0,2)}\right)\le \Pr_{\operatorname{free}}^{(0,2),(\g_t(0),\g_t(2))}\left({\m{Fall}}_{\gamma, s}\right)$$
	where $\Pr_{\operatorname{free}}^{(a,b),(y,z)}$ denotes the law of Brownian bridge on $[a,b]$ starting at $y$ and ending at $z$.  Using this we have 
	\begin{align}\label{e:fall2}
		\Pr\left(\m{Fall}_{\gamma, s}, \m{G}_{\gamma, s}\right)&\le \Pr\left(\m{Fall}_{\gamma, s}, \g_t(0)\le \tfrac{s}4, \g_t(2)+m_t(2)\ge -\tfrac{s}{4}\right)\notag \\ &\le  \Ex\left[\ind_{\g_t(0) \le \frac{s}4}\cdot \ind_{\g_t(2)+m_t(2)\ge -\frac{s}{4}}\Pr_{\m{free}}^{(0,2), (\g_t(0), \g_t(2))}\left({\m{Fall}}_{\gamma, s}\right)\right]\notag \\& \le \sup\bigg\{\Pr_{\m{free}}^{(0,2),y,z}\left({\m{Fall}}_{\gamma, s}\right) : y \le \tfrac{s}4, z+m_t(2) \ge -\tfrac{s}{4}\bigg\} \notag \\ & =\Pr_{\m{free}}^{(0,2),s/4,-s/4-m_t(2)}\left({\m{Fall}}_{\gamma, s} \right).
	\end{align}
	Next, we write the final term in \eqref{e:fall2} as $$\Pr_{\m{free}}^{(0,2),s/4,-s/4-m_t(2)}\left({\m{Fall}}_{\gamma, s}\right) = \Pr\left(\inf_{z \in [0, \gamma]}\bigg\{B'(z) +m_t(z) \bigg\}\le -s\gamma^{1/2}\right)$$ where $B':[0,2]\rightarrow \R$ is a Brownian bridge with $B'(0) = 0$ and $B'(2) = -m_t(2)- \frac{s}{2}.$ Now, set $B(z): = B'(z) - \frac{z}{2}(-m_t(2)-\frac{s}{2}).$ Then $B$ is a Brownian bridge with $B(0) = B(2) = 0$ and we obtain 
	\begin{align}
		\Pr\left(\inf_{z\in [0, \gamma]} (B'(z) + m_t(z))\le - s\gamma^{1/2}\right) &\le \Pr\left(\inf_{z\in [0, \gamma]} B(z) \le - s\gamma^{1/2}- \tfrac{\gamma}{2}(-m_t(2)- \tfrac{s}{2})\right) \notag \\&\le \Pr\left(\inf_{z\in [0,\gamma]} B(z) \le - \tfrac{s}{2}\gamma^{1/2}\right).
	\end{align}
	The latter inequality is due to $\gamma^{1/2}(m_t(2)+\tfrac{s}{2}) \le s$ as $s\ge 100$ and $\gamma \le \sqrt{t}$. The right-hand probability can be estimated via Brownian calculations, which yields the desired bound of the form $\Con\exp(-\frac1{\Con} s^2)$.
	
	\medskip
	
	\noindent We next prove the $\m{Rise}_{\gamma,s}$ bound in \eqref{e:bounds2}. Note that $\sup_{z\in [0,\gamma]} m_t(z) \le \frac{\gamma^2}{t^{1/4}} \le \frac12s\gamma^{1/2}$ (as $\gamma\le \sqrt{t}\le 1$ and $s\ge 4$). Thus it suffices to show
	\begin{align}\label{tsh}
		\Pr\left(\m{Rise}_{\gamma,s}^{(1)},\m{G}_{\gamma,s}\right)\le \Con\exp\left(-\tfrac1\Con s^2\right), \quad \m{Rise}_{\gamma,s}^{(1)}:=\bigg\{\sup_{z\in [0, \gamma]}\g_t(z)\ge \g_t(0) + \tfrac12s\sqrt\gamma\bigg\}.
	\end{align}
	Set $$\chi:= \inf\bigg\{ x\in (0, \gamma] \mid \g_t(x)- \g_t(0) \ge \tfrac12 s\gamma^{1/2} \bigg\},$$
	and set $\chi=\infty$ if  no such points exist. Then we have $\Pr\left(\m{Rise}^{(1)}_{\gamma, s}, \m{G}_{\gamma, s}\right) = \Pr\left(\chi \le \gamma, \m{G}_{\gamma,s}\right)$ and we can write the right-hand probability as \begin{align}\label{eq:rise2}
		\Pr\left(\chi \le \gamma, \m{G}_{\gamma,s},\g_t(\chi)- \g_t(\gamma)<\tfrac14s\sqrt{\gamma}\right)+ \Pr\left(\chi \le \gamma, \m{G}_{\gamma,s},\g_t(\chi)- \g_t(\gamma)\ge\tfrac14s\sqrt{\gamma} \right).
	\end{align} 
	On the event $\{\chi \le \gamma, \m{G}_{\gamma,s},\g_t(\chi)- \g_t(\gamma)<\tfrac14s\sqrt{\gamma}\}$ we have that $\{\g_t(\gamma) - \g_t(0) \ge \frac14s\sqrt\gamma\}$ holds as the continuity of $\g_t(\cdot)$ implies that $\g_t(\chi)=\g_t(0)+\frac12s\sqrt{\gamma}$ on $\{\chi\le \gamma\}$ event. Now with the same argument of the $\m{Fall}_{\gamma,s}$ event, we bound the probability of this occurrence by $\Con\exp(-\frac1\Con s^2)$ for some constant $\Con>0$. This is why $\m{G}_{\gamma, s}$ involves $\g_t(-2 + \gamma)$ and $\g_t(\gamma).$ The parabolic term $m_t(z)$ again can be ignored as $\sup_{z\in [0,\gamma]} m_t(z) \le \gamma^{3/2} \le \frac18 s\gamma^{1/2}$ for $s\ge 8$.

	Let us focus on the second term in \eqref{eq:rise2}. Note that $(\chi,2)$ is a stopping domain and $\{\g_t(\chi)-\g_t(\gamma) \ge \frac14s\sqrt{\gamma}\}$ is a monotone event w.r.t.~$[\chi,2]$. Applying Lemma \ref{l:mon} one has
	\begin{align*}
		\Pr\left(\g_t(\chi)- \g_t(\gamma)\ge\tfrac14s\sqrt{\gamma} \mid (\g_t(x))_{x\notin (\chi,2)}\right) \le \Pr_{\operatorname{free}}^{(\chi,2),(\g_t(\chi),\g_t(2))}\left(\g_t(\chi)- \g_t(\gamma)\ge\tfrac14s\sqrt{\gamma} \right).
	\end{align*}	
	Note that on $\{\chi\le \gamma, \m{G}_{\gamma, s}\}$  we have
	{\begin{align}
	    |\g_t(\chi)-\g_t(2)| =|\g_t(0)+\tfrac12s\sqrt{\gamma}-\g_t(2)| \le s/4+\tfrac12s\sqrt{\gamma} +m_t(2)+s/4=s+m_t(2).
	\end{align}
	As $2-\chi \ge 1$ on $\{\chi \le \gamma\}$, we thus get that the absolute value of the slope of the linearly interpolated line joining $(\chi,\g_t(\chi))$ and $(2,\g_t(2))$ is at most $s+m_t(2)$. Note that $\frac14s\sqrt{\gamma} \ge \gamma(s+m_t(2))$ due to \eqref{onte}. Thus the event $\{\g_t(\chi)-\g_t(\gamma) \ge \frac14s\sqrt{\gamma}\}$ entails that the $\g_t(\gamma)$ lies below the linearly interpolated line.} Under Brownian law, this has probability $1/2$.	Thus,
	\begin{align*}
		&\Pr\left(\chi \le \gamma, \m{G}_{\gamma,s},\g_t(\chi)-\g_t(\gamma)\ge \tfrac14s\sqrt\gamma \right)\\ & \le \Ex\left[\ind_{\chi\le \gamma, \m{G}_{\gamma, s}}\Pr_{\operatorname{free}}^{(\chi,2),(\g_t(\chi),\g_t(2))}\left(\g_t(\chi)- \g_t(\gamma)\ge\tfrac14s\sqrt{\gamma} \right)\right] \le \tfrac{1}{2}\Ex\left[\ind_{\chi \le \gamma, \m{G}_{\gamma, s}}\right].
	\end{align*}
	Hence we have shown that $\Pr\left(\chi \le \gamma, \m{G}_{\gamma, s}\right)\le \Con\exp(-\frac1\Con s^2) + \frac{1}{2}\Pr\left(\chi \le \gamma, \m{G}_{\gamma, s}\right)$ which implies that $\Pr\left(\chi \le \gamma, \m{G}_{\gamma, s}\right) \le 2\Con\exp(-\frac1\Con s^2)$ which gives us the bound in \eqref{tsh}, completing the proof.
\end{proof}

\section{Modulus of Continuity for rescaled $\cdrp$ measures}\label{sec:moc}

The main goal of this section is to establish quenched modulus of continuity estimates: Proposition \ref{p:qmc} and Proposition \ref{p:qmc.ptl}, for $\cdrp$ measures under long-time scalings. The proof of these propositions requires detailed study of the tail probabilities of two-point difference when scaled according to long-time. This is conducted in Proposition \ref{p:ltd} and Proposition \ref{p:ltd.ptl} respectively. One of the key technical inputs in the proofs of Propositions \ref{p:ltd} and \ref{p:ltd.ptl} is a parabolic decay estimate of a multivariate spatial process involving several long-time KPZ sheets. This estimate appears in Lemma \ref{l:ltail} and is proved in Section \ref{sec3.1}. In the following text, we first state those Propositions \ref{p:ltd} and \ref{p:ltd.ptl} and assuming their validity, we state and prove the modulus of continuity estimates. Proofs of Proposition \ref{p:ltd} and \ref{p:ltd.ptl} are deferred to Section \ref{sec:twopt}.   

\begin{proposition}[Long-time two-point difference] \label{p:ltd} Fix any $\e\in (0,1]$, $\delta\in (0,\frac12)$, and $\tau\ge 1$. Take $x\in [-\tau\e^{-\frac23},\tau\e^{-\frac23}]$. Let $V \sim \cdrp(0,0;x,\e^{-1})$. For $t\in [0,1]$, set $L_t^{(\e)}:=\e^{\frac23}V({\e^{-1}t})$. There exist two absolute constants $\Con_1(\tau,\delta)>0$ and $\Con_2(\tau,\delta)>0$ such that for all $m\ge 1$ and $t\neq s\in [0,1]$ we have
	\begin{align*}
		\Pr\left[\Pr^{\xi}(|L_s^{(\e)}-L_t^{(\e)}|\ge m|s-t|^{\frac12-\delta}) \ge \Con_1\exp(-\tfrac1{\Con_1}m^2)\right] \le \Con_2\exp\left(-\tfrac1{\Con_2}m^3\right).
	\end{align*}
\end{proposition}

We have the following point-to-line analogue.

\begin{customprop}{\ref*{p:ltd}-(point-to-line)}
	\label{p:ltd.ptl}
	Fix any $\e\in (0,1]$, $\delta\in (0,\frac12)$. Let $V \sim \cdrp(0,0;*,\e^{-1})$. For $t\in [0,1]$, set $L_{t,*}^{(\e)}:=\e^{\frac23}V({\e^{-1}t})$. There exist two absolute constants $\Con_1(\delta)>0$ and $\Con_2(\delta)>0$ such that for all $m\ge 1$ and $t\neq s\in [0,1]$ we have
	\begin{align*}
		\Pr\left[\Pr_*^{\xi}(|L_{s,*}^{(\e)}-L_{t,*}^{(\e)}|\ge m|s-t|^{\frac12-\delta}) \ge \Con_1\exp(-\tfrac1{\Con_1}m^2)\right] \le \Con_2\exp\left(-\tfrac1{\Con_2}m^3\right).
	\end{align*}
\end{customprop}

\begin{remark} In the above propositions, the quenched probability of the tail event of two-point difference of rescaled polymers is viewed as a random variable.  The above propositions provide quantitative decay estimates of this random variable being away from zero {for point-to-point and point-to-line polymers under long-time regime}.	  
\end{remark}

\begin{proposition}[Quenched Modulus of Continuity]  \label{p:qmc} Fix $\e\in (0,1]$, $\delta\in (0,\frac12)$ and $\tau\ge 1$. Take $y\in [-\tau\e^{-\frac23},\tau\e^{-\frac23}]$. Let $V\sim \cdrp(0,0;y,\e^{-1})$. Set $L_t^{(\e)}:=\e^{\frac23}V({\e^{-1}t})$ for $t\in [0,1]$. Then there exist two constants $\Con_1(\tau,\delta)>0$ and $\Con_2(\tau,\delta)>0$ such that for all $m\ge 1$ we have
	\begin{align}\label{e:lqmc}
		& \Pr\bigg[\Pr^{\xi}\bigg(\sup_{\substack{t\neq s\in [0,1]}} \frac{|L_s^{(\e)}-L_t^{(\e)}|}{|t-s|^{\frac12-\delta}\log\frac{2}{|t-s|}} \ge m\bigg) \ge \Con_1\exp\left(-\tfrac1{\Con_1}m^2\right)\bigg]  \le \Con_2\exp\left(-\tfrac1{\Con_2}m^3\right). 
	\end{align}
\end{proposition}

\begin{customprop}{\ref*{p:qmc}-(point-to-line)}\label{p:qmc.ptl} Fix $\e\in (0,1]$, $\delta\in (0,\frac12)$. Let $V\sim \cdrp(0,0;*,\e^{-1})$. For $t\in [0,1]$, set  $L_{t,*}^{(\e)}:=\e^{\frac23}V({\e^{-1}t})$. Then there exist two constants $\Con_1(\delta)>0$ and $\Con_2(\delta)>0$ such that for all $m\ge 1$ we have
	\begin{align*}
		& \Pr\bigg[\Pr_*^{\xi}\bigg(\sup_{\substack{t\neq s\in [0,1]}} \frac{|L_{s,*}^{(\e)}-L_{t,*}^{(\e)}|}{|t-s|^{\frac12-\delta}\log\frac{2}{|t-s|}} \ge m\bigg) \ge \Con_1\exp\left(-\tfrac1{\Con_1}m^2\right)\bigg]  \le \Con_2\exp\left(-\tfrac1{\Con_2}m^3\right). 
	\end{align*}
\end{customprop}

\begin{remark} The paths of continuum directed random polymer are known to be H\"older continuous with exponent $\gamma$, for every $\gamma<1/2$ (see \cite[Theorem 4.3]{akq}).  Our Theorem \ref{p:qmc} corroborates this fact by giving quantitative tail bounds to the quenched modulus of continuity.
\end{remark}

Before proving Propositions \ref{p:qmc} and \ref{p:qmc.ptl}, we present below a few important corollaries for point-to-point long-time polymer. Similar corollaries hold for point-to-line case as well.

\begin{corollary}\label{l:lqmcc1}
	Fix $\e\in (0,1]$, and $\tau\ge 1$. Take $x\in [-\tau\e^{-\frac23},\tau\e^{-\frac23}]$. Let $V\sim \cdrp(0,0;x,\e^{-1})$. For $t\in [0,1]$ set  $L_t^{(\e)}:=\e^{\frac23}V({\e^{-1}t})$. Then there exist two constants $\Con_1(\tau)>0$ and $\Con_2(\tau)>0$ such that for all $m\ge 1$ we have
	\begin{align}\label{e:lqmcc1}
		& \Pr\bigg[\Pr^{\xi}\bigg(\sup_{\substack{t\in [0,1]}}|L_t^{(\e)}| \ge m\bigg) \ge \Con_1\exp\left(-\tfrac1{\Con_1}m^2\right)\bigg]  \le \Con_2\exp\left(-\tfrac1{\Con_2}m^3\right). 
	\end{align}
\end{corollary}
\begin{proof}
	Set $s = 0$ and $\rho=1+\sup_{t\in (0,1]} t^{1/4}\log \frac2t \in (1,\infty)$. By Proposition \ref{p:qmc}, with $\delta=\frac14$ there exist $\Con_1(\tau)$ and $\Con_2(\tau)$ such that for all $m \ge 1$, \eqref{e:lqmc} holds with $s = 0$. Replacing $m$ with $m/\rho$ in \eqref{e:lqmc} yields that 
	\begin{align*}
		&\Pr\bigg[\Pr^{\xi}\bigg(\sup_{\substack{t\in [0,1]}}|L_t^{(\e)}|\ge m\bigg) \ge \Con_1\exp\left(-\tfrac1{\Con_1}m^2\right)\bigg] \\ & \le \Pr\bigg[\Pr^{\xi}\bigg(\sup_{\substack{t\in [0,1]}} \frac{|L_t^{(\e)}|}{t^{\frac14}\log\frac2t} \ge \frac{m}{\rho}\bigg) \ge \Con_1\exp\left(-\tfrac1{\Con_1}(\tfrac{m}{\rho})^2\right)\bigg]  \le \Con_2\exp\left(-\tfrac1{\Con_2}(\tfrac{m}{\rho})^3\right). 
	\end{align*}
	Adjusting $\Con_2$ further we get the desired result.
\end{proof}
From Proposition \ref{p:qmc}, we also obtain the annealed modulus of continuity. 
\begin{corollary}[Annealed Modulus of Continuity]  \label{p:amc} Fix $\e\in (0,1]$, $\delta\in (0,\frac12)$ and $\tau\ge 1$. Take $y\in [-\tau\e^{-\frac23},\tau\e^{-\frac23}]$. Let  $V\sim \cdrp(0,0;y,\e^{-1})$. Set $L_t^{(\e)}:=\e^{\frac23}V({\e^{-1}t})$ for $t\in [0,1]$. Then there exists a constant $\Con(\tau,\delta)>0$ such that for all $m\ge 1$ we have
	\begin{align}\label{e:lamc}
		& \Pr\bigg(\sup_{\substack{t\neq s\in [0,1]}} \frac{|L_s^{(\e)}-L_t^{(\e)}|}{|t-s|^{\frac12-\delta}\log\frac{2}{|t-s|}} \ge m\bigg) \le \Con\exp\left(-\tfrac1{\Con}m^2\right). 
	\end{align}
\end{corollary}
Clearly one has similar corollaries for the point-to-line version which follow from Proposition \ref{p:qmc.ptl} instead. For brevity, we do not record them separately. We now assume Proposition \ref{p:ltd} (Proposition \ref{p:ltd.ptl}) and complete the proof of Proposition \ref{p:qmc} (Proposition \ref{p:qmc.ptl}). 

\begin{proof}[Proof of Propositions \ref{p:qmc} and \ref{p:qmc.ptl}] Fix $\tau\ge 1$ and $m\ge 16\tau^2+1$. The main idea is to mimic Levy's proof of modulus of continuity of Brownian motion. Since our proposition deals with quenched versions, we keep the proof here for the sake of completeness. We only prove \eqref{e:lqmc} using Proposition \ref{p:ltd}. Proof of Proposition \ref{p:qmc.ptl} follows in a similar manner using Proposition \ref{p:ltd.ptl}. To prove \eqref{e:lqmc}, we first control the modulus of continuity on dyadic points of $[0,1]$. Fix $\delta>0$ and set $\gamma=\frac12-\delta$. Define
	\begin{align*}
		\norm{L^{(\e)}}_n:=\sup_{k=\{1,\ldots,2^n\}} \left|L_{k2^{-n}}^{(\e)}-L_{(k-1)2^{-n}}^{(\e)}\right|, \quad \norm{L^{(\e)}}:=\sup_{n\ge 0} \frac{\norm{L^{(\e)}}_n 2^{n\gamma}}{n+1}.
	\end{align*}
	Observe that by union bound
	\begin{align*}
		\Pr^{{\xi}}\left(\norm{L^{(\e)}} \ge m\right) \le \sum_{n=0}^{\infty}\sum_{k=1}^{2^n}\Pr^{\xi}\left(\left|L_{k2^{-n}}^{(\e)}-L_{(k-1)2^{-n}}^{(\e)}\right| \ge m 2^{-n\gamma}(n+1)\right).
	\end{align*}
	Thus in light of Proposition \ref{p:ltd} we see that with probability {at least}
	$$1-\sum_{n=0}^{\infty}\Con_2 2^n\exp\left(-\tfrac1{\Con_2}m^3(n+1)^3\right) \ge 1-\Con_2'\exp\left(-\tfrac1{\Con_2'}m^3\right)$$
	we have
	\begin{align*}
		\Pr^{\xi}\left(\norm{L^{(\e)}} \ge m\right) \le \sum_{n=0}^{\infty}\Con_1 2^n \exp\left(-\tfrac1\Con_1m^2(n+1)^2\right) \le \Con_1' \exp\left(-\tfrac1{\Con_1'}m^2\right) .
	\end{align*}
	Finally one can extend the results to all points by continuity of $L^{(\e)}$ and observing the following string of inequalities that holds deterministically.	For any $0\le s < t\le 1$ we have
	\begin{align}\label{arekt}
		|L_t^{(\e)}-L_s^{(\e)}| & \le \sum_{n=1}^{\infty}  \left|L^{(\e)}_{2^{-n}\lfloor {2^nt}\rfloor}-L^{(\e)}_{2^{-n+1}\lfloor {2^{n-1}t}\rfloor}-L^{(\e)}_{2^{-n}\lfloor {2^ns}\rfloor}+L^{(\e)}_{2^{-n+1}\lfloor {2^{n-1}s}\rfloor}\right|. 
	\end{align}
	{Note that we have
	\begin{align*}
	    & \left|L^{(\e)}_{2^{-n}\lfloor {2^nt}\rfloor}-L^{(\e)}_{2^{-n+1}\lfloor {2^{n-1}t}\rfloor}-L^{(\e)}_{2^{-n}\lfloor {2^ns}\rfloor}+L^{(\e)}_{2^{-n+1}\lfloor {2^{n-1}s}\rfloor}\right| \\ & \le \left|L^{(\e)}_{2^{-n}\lfloor {2^nt}\rfloor}-L^{(\e)}_{2^{-n+1}\lfloor {2^{n-1}t}\rfloor}\right|+\left|L^{(\e)}_{2^{-n}\lfloor {2^ns}\rfloor}-L^{(\e)}_{2^{-n+1}\lfloor {2^{n-1}s}\rfloor}\right| \le 2\norm{L^{(\e)}}_n,
	\end{align*}
	and
	\begin{align*}
	    & \left|L^{(\e)}_{2^{-n}\lfloor {2^nt}\rfloor}-L^{(\e)}_{2^{-n+1}\lfloor {2^{n-1}t}\rfloor}-L^{(\e)}_{2^{-n}\lfloor {2^ns}\rfloor}+L^{(\e)}_{2^{-n+1}\lfloor {2^{n-1}s}\rfloor}\right| \\ & \le \left|L^{(\e)}_{2^{-n}\lfloor {2^nt}\rfloor}-L^{(\e)}_{2^{-n}\lfloor {2^ns}\rfloor}\right|+\left|L^{(\e)}_{2^{-n+1}\lfloor {2^{n-1}t}\rfloor}-L^{(\e)}_{2^{-n+1}\lfloor {2^{n-1}s}\rfloor}\right| \le 2(t-s)2^n\norm{L^{(\e)}}_n.
	\end{align*}
	Combining the above two inequalities we get}
	\begin{align*}
	    	\mbox{r.h.s.~of \eqref{arekt}} & \le \sum_{n=1}^{\infty} 2\left({|t-s|}2^n\wedge 2\right)\norm{L^{(\e)}}_n \\ & \le \norm{L^{(\e)}}\sum_{n=1}^{\infty} (n+1)2^{-n\gamma}\left({|t-s|}2^n\wedge 2\right) \le c_2\norm{L^{(\e)}} \cdot |t-s|^{\gamma}\log \tfrac{2}{|t-s|}.
	\end{align*}

	where $c_2>0$ is an absolute constant. Combining this with the bound for $\Pr^{\xi}(\norm{L^{(\e)}}\ge m)$, completes the proof. 
\end{proof}

\subsection{Tail bounds for multivariate spatial process} \label{sec3.1} {Recall the KPZ sheet $\mathcal{H}(\cdot,\cdot;\cdot,\cdot)$ defined in \eqref{calhd}}. The core idea behind the proof of Propositions \ref{p:ltd} and \ref{p:ltd.ptl} is to establish parabolic decay estimates of sum of several KPZ sheets scaled according to long-time. We record this parabolic decay estimate in the following Lemma \ref{l:ltail}.

\begin{lemma}[Long-time multivariate spatial process tail bound]\label{l:ltail} Fix any $k \in \Z_{>0}$ and $\nu\in (0,1)$. Set $x_0=0$, and $\vec{x}:=(x_1,\ldots,x_k)$. For any $\e\in (0,1)$ consider $0 = t_0 < t_1 < \dots < t_k = 1$. Set $\vec{t}:=(t_1,\ldots,t_k)$. Then there exists a constant $\Con = \Con(k,\nu) $ such that for all $s > 0$ we have
	\begin{align}\label{eq:ltail}
		\Pr\left(\sup_{\vec{x} \in \R^k} \left[F_{\vec{t};\e}(\vec{x})+\sum_{i=0}^{k-1} \frac{(1-\nu)(x_{i+1}-x_i)^2}{2(t_{i+1}-t_i)}\right] \ge s\right)\le \Con \exp\left(-\tfrac{1}{\Con}s^{3/2}\right).
	\end{align}
	where
	\begin{equation}
		\begin{aligned}\label{def:f}
			F_{\vec{t};\e}(\vec{x}) & :  = \e^{1/3}\sum_{i=0}^{k-1} \left[ \calH({x_i\e^{-2/3},\e^{-1}t_i;x_{i+1}\e^{-2/3},\e^{-1}t_{i+1}})+\tfrac{\e^{-1}(t_{i+1}-t_i)}{24}\right. \\ & \hspace{5cm}\left.+\ind\{t_{i+1}-t_i\le \e\} \cdot \log \sqrt{2\pi \e^{-1}(t_{i+1}-t_i)}\right]. 
		\end{aligned}   
	\end{equation}
\end{lemma}
\begin{proof} For clarity, we split the proof into three steps.
	
	\medskip
	
	\noindent\textbf{Step 1.}  Let us fix any $\e\in (0,1)$ consider $0 = t_0 < t_1 < \dots < t_k = 1$. For brevity,  we denote $F(\vec{x}):=F_{\vec{t};\e}(\vec{x})$ and set
	\begin{align}\label{barf}
		\bar{F}(\vec{x}): =  F(\vec{x}) + \sum_{i=0}^{k-1} \frac{(x_{i+1}-x_i)^2}{2(t_{i+1}-t_i)}.
	\end{align}
	For any $\vec{a}=(a_1,\ldots,a_k)\in \Z^k$, set $V_{\vec{a}}:=[a_1,a_1+1]\times\cdots\times[a_{k},a_k+1]$ and set
	\begin{align}\label{norma}
		\norm{\vec{a}}^2:=a_1^2+\min_{\vec{x}\in V_{\vec{a}}}\sum_{i=1}^{k-1} (x_{i+1}-x_i)^2.
	\end{align}
	We claim that for any $\vec{a} = (a_1, \ldots, a_k) \in \Z^k$ and $\nu \in (0,1)$
	\begin{align}\label{eq:pabd}
		\Pr\left(\sup_{\vec{x}\in V_{\vec{a}}}\left[F(\vec{x})+\sum_{i=0}^{k-1} \frac{(1-\nu)(x_{i+1}-x_i)^2}{2(t_{i+1}-t_i)} \right]\ge s \right)\le \Con \exp\left(-\tfrac{1}{\Con}(s^{3/2}+\norm{\vec{a}}^3)\right)
	\end{align}
	for some $\Con = \Con(k,\nu)>0.$ Assuming \eqref{eq:pabd} by union bound we obtain 
	\begin{align*}
		\text{ l.h.s of  \eqref{eq:ltail}} & = \Pr\left(\sup_{\vec{x}\in \R^k}\left[F(\vec{x})+\sum_{i=0}^{k-1} \frac{(1-\nu)(x_{i+1}-x_i)^2}{2(t_{i+1}-t_i)}\right] \ge s \right)\\&\le \sum_{\vec{a} \in \Z^k} 	\Pr\left(\sup_{\vec{x}\in V_{\vec{a}}}\left[F(\vec{x})+\sum_{i=0}^{k-1} \frac{(1-\nu)(x_{i+1}-x_i)^2}{2(t_{i+1}-t_i)} \right]\ge s \right) \\ &  \le \sum_{\vec{a}\in \Z^k} \Con \exp\left(-\tfrac{1}{\Con}(s^{3/2}+\norm{\vec{a}}^3)\right).
	\end{align*}
	The r.h.s.~of the above display is upper bounded by $\Con \exp(-\frac{1}{\Con}s^{3/2})$ and proves \eqref{eq:ltail}. Thus it suffices to verify \eqref{eq:pabd} in the rest of the proof. 
	
	\medskip

	\noindent\textbf{Step 2.} In this step, we prove the claim in \eqref{eq:pabd}. Note that 
	$$\Pr\left(\sup_{\vec{x}\in V_{\vec{a}}}\left[F(\vec{x})+\sum_{i=0}^{k-1} \frac{(1-\nu)(x_{i+1}-x_i)^2}{2(t_{i+1}-t_i)}\right]\ge s \right)  \le \Pr\left(\sup_{\vec{x}\in V_{\vec{a}}}|\bar{F}(\vec{x})| \ge s + \tfrac\nu2\norm{\vec{a}}^2\right)$$ 
	by the definition of $\bar{F}(\cdot)$ in \eqref{barf} and the defintion of $\norm{\vec{a}}^2$ from \eqref{norma}. Applying union bound yields
	\begin{align}\label{eq:intbd}
		&\Pr\left(\sup_{\vec{x}\in V_{\vec{a}}}|\bar{F}(\vec{x})| \ge s + \tfrac\nu2\norm{\vec{a}}^2\right) \notag\\ & \le \Pr\left(\sup_{\vec{x}\in V_{\vec{a}}}|\bar{F}(\vec{x}) - \bar{F}(\vec{a})| \ge \tfrac{s}2 + \tfrac{\nu}4 \norm{\vec{a}}^2\right)  + \Pr\left(|\bar{F}(\vec{a})| \ge \tfrac{s}2 + \tfrac{\nu}4 \norm{\vec{a}}^2\right) .
	\end{align}
	In the rest of the proof, we bound both summands on the r.h.s of \eqref{eq:intbd} from above by $\Con  \exp(-\frac{1}{\Con}(s^{3/2}+\norm{\vec{a}}^3))$ individually. To control the first term, we first need an a priori estimate. We claim that for all $u\in [0,1]$, $i=1,2,\ldots,k$ and $s>0$ we have
	\begin{align}\label{eq:diff}
		\Pr\left(\bar{F}(\vec{a}+e_i \cdot u)-\bar{F}(\vec{a}) \ge su^{1/4}\right) \le \Con\exp\left(-\tfrac1\Con s^{3/2}\right).
	\end{align}
	for some absolute constant $\Con>0$. We will prove \eqref{eq:diff} in the next step. Given \eqref{eq:diff}, appealing to Lemma 3.3 in \cite{dv18} with $\alpha = \alpha_i = \frac{1}{4}, \beta = \beta_i = \frac{3}{2}, r = r_i = 1$, we get that for all $m>0$
	\begin{align*}
		\Pr\left(\sup_{\vec{x}\in V_{\vec{a}}} |\bar{F}(\vec{x})-\bar{F}(\vec{a})| \ge m\right) \le \Con\exp\left(-\tfrac1\Con m^{3/2}\right).
	\end{align*}
	Taking $m=\frac{s}{2}+\frac{\nu}4\norm{\vec{a}}^2$ in above, this yields the desired estimate for the first term in \eqref{eq:intbd}. 	
	
	For the second term in \eqref{eq:intbd}, via the definition of $\bar{F}$ in \eqref{barf} applying union bounds we have 
	\begin{align} \notag
		&\Pr\left(|\bar{F}(\vec{a})| \ge \tfrac{s}4 + \tfrac{\nu}4\norm{\vec{a}}^2\right)  \\ & \le \sum_{i =0}^{k-1} \Pr\left(\left| \e^{1/3} \calH({a_i\e^{-2/3},\e^{-1}t_i;a_{i+1}\e^{-2/3},\e^{-1}t_{i+1}})+\tfrac{\e^{-2/3}(t_{i+1}-t_i)}{24}\right.\right. \notag \\ & \left.\left. \hspace{3cm}
		+\tfrac{(a_{i+1}-a_i)^2}{2(t_{i+1}-t_i)}+\e^{1/3}\ind\{t_{i+1}-t_i\le \e\} \cdot \log \sqrt{2\pi \e^{-1}(t_{i+1}-t_i)}\right| \ge \tfrac{s}{4k} + \tfrac{\nu}{4k}\norm{\vec{a}}^2\right) \notag \\ & \le \sum_{i =0}^{k-1} \Pr\left(\left| \e^{1/3} \calH(0,\e^{-1}(t_{i+1}-t_i))+\tfrac{\e^{-2/3}(t_{i+1}-t_i)}{24}\right.\right. \notag \\ & \hspace{3cm}\left.\left.+\e^{1/3}\ind\{t_{i+1}-t_i\le \e\} \cdot \log \sqrt{2\pi \e^{-1}(t_{i+1}-t_i)}\right| \ge \tfrac{s}{4k} + \tfrac{\nu}{4k}\norm{\vec{a}}^2\right) \label{e:kcalh}
	\end{align}
	where the last line follows from stationarity of {the shifted version of} $\calH$. Now if $\e^{-1}(t_{i+1}-t_i)> 1$, we may use long-time scaling to get
	\begin{align*}
		\e^{1/3} \calH(0,\e^{-1}(t_{i+1}-t_i))+\frac{\e^{-2/3}(t_{i+1}-t_i)}{24}=\frac{\h_{\e^{-1}(t_{i+1}-t_i)}(0)}{(t_{i+1}-t_i)^{-1/3}}.
	\end{align*}
	{Using the fact that $\e < |t_{i+1}-t_i|\le 1$ along with the one-point long-time tail estimates from Proposition \ref{p:ltkpzeq} \ref{p:ltail} we get 
	\begin{align*}
	    \Pr\left(|\h_{\e^{-1}(t_{i+1}-t_i)}(0)| \ge {(t_{i+1}-t_i)^{-1/3}}(\tfrac{s}{4k}+\tfrac{\nu}{4k}\norm{\vec{a}}^2)\right) & \le \Pr\left(|\h_{\e^{-1}(t_{i+1}-t_i)}(0)| \ge \tfrac{s}{4k}+\tfrac{\nu}{4k}\norm{\vec{a}}^2\right) \\ & \le  \Con \exp(-\tfrac1\Con(s + \norm{\vec{a}}^2)^{3/2}) \\ & \le \Con  \exp\left(-\tfrac{1}{\Con}(s^{3/2}+\norm{\vec{a}}^3)\right),
	\end{align*}
	for some constant $\Con=\Con(k,\nu)>0$. If $\e^{-1}(t_{i+1}-t_i)\le 1$, we may use short-time scaling to get
	\begin{align*}
		& \e^{1/3} \calH(0,\e^{-1}(t_{i+1}-t_i))+\frac{\e^{-2/3}(t_{i+1}-t_i)}{24}+ \e^{1/3}\log \sqrt{2\pi \e^{-1}(t_{i+1}-t_i)} \\ & \hspace{2cm}=\e^{1/3}\big(\tfrac{\pi \e^{-1}(t_{i+1}-t_i)}4\big)^{1/4}\g_{\e^{-1}(t_{i+1}-t_i)}(0)+\frac{\e^{-2/3}(t_{i+1}-t_i)}{24}.
	\end{align*}
	The linear term above is uniformly bounded in this case. Furthermore, $$\e^{1/3}\big(\tfrac{\pi \e^{-1}(t_{i+1}-t_i)}4\big)^{1/4}=\big(\tfrac{\pi (t_{i+1}-t_i)}4\big)^{1/4}\e^{1/12}\le 1.$$} Thus, in this case, appealing to one-point short-time tail estimates from Proposition \ref{p:stkpzeq} \ref{p:stail}, we have
	\begin{align*}
		\mbox{r.h.s.~of \eqref{e:kcalh}} \le \Con \exp(-\tfrac1\Con(s + \norm{\vec{a}}^2)^{3/2}) \le \Con  \exp\left(-\tfrac{1}{\Con}(s^{3/2}+\norm{\vec{a}}^3)\right)
	\end{align*}
	for some constant $\Con = \Con (k,\nu)>0.$  
	
	This proves the required bound for the second term in \eqref{eq:intbd}. Combining the bounds for the two terms in \eqref{eq:diff}, we thus arrive at \eqref{eq:pabd}. Hence, all we are left to show is \eqref{eq:diff} which we do in the next step.
	
	\medskip
	
	\noindent\textbf{Step 3.} Fix $\vec{a}\in \Z^k$, fix $i=1,2,\ldots, k$. The goal of this step is to show \eqref{eq:diff}. Towards this end, note that for each coordinate vector $e_i, i = 1, \ldots, k-1$, and for $u\in [0,1]$ observe that
	\begin{align*}
		& \bar{F}(\vec{a} + e_i\cdot u) - \bar{F}(\vec{a}) \\ & = \e^{1/3}\left[\calH(a_{i-1}\e^{-2/3},\e^{-1}t_{i-1};(a_i+u)\e^{-2/3},\e^{-1}t_i)-\calH(a_{i-1}\e^{-2/3},\e^{-1}t_{i-1};a_i\e^{-2/3},\e^{-1}t_i)\right]  \\ & \hspace{4cm}+\frac{(a_{i-1}-a_i-u)^2-(a_{i-1}-a_i)^2}{2(t_{i}-t_{i-1})} \\ & \hspace{1cm}+ \e^{1/3}\left[\calH((a_{i}+u)\e^{-2/3},\e^{-1}t_{i};a_{i+1}\e^{-2/3},\e^{-1}t_{i+1})-\calH(a_{i}\e^{-2/3},\e^{-1}t_{i};a_{i+1}\e^{-2/3},\e^{-1}t_{i+1})\right] \\ & \hspace{5cm}+\frac{(a_{i+1}-a_i-u)^2-(a_{i+1}-a_i)^2}{2(t_{i+1}-t_i)}. 
	\end{align*}
	Thus using distributional identities (see Remark \ref{re:gscal}) by union bound for all $s>0$ we get that
	\begin{align}
		\notag	& \Pr\left(|\bar{F}(\vec{a} + e_i\cdot u) - \bar{F}(\vec{a})| \ge s u^{\frac14}\right) \\ &\le \Pr\left(\e^{\frac13}\left|\bar\calH_{\e^{-1}(t_{i}-t_{i-1})}((a_i+u-a_{i-1})\e^{-2/3})-\bar\calH_{\e^{-1}(t_{i}-t_{i-1})}((a_i-a_{i-1})\e^{-2/3})\right|\ge\tfrac{s}2u^{\frac14}\right) \label{A} \\ & \hspace{0.5cm}+ \Pr\left(\e^{\frac13}\left|\bar\calH_{\e^{-1}(t_{i+1}-t_{i})}((a_i+u-a_{i+1})\e^{-2/3})-\bar\calH_{\e^{-1}(t_{i+1}-t_{i})}((a_i-a_{i+1})\e^{-2/3})\right|\ge\tfrac{s}2u^{\frac14}\right) \label{B},
	\end{align}
	where $\bar\calH_t(x):=\calH(x,t)+\frac{x^2}{2t}.$ We now proceed to bound the second term on the r.h.s.~of above display (that is the term in \eqref{B}); the bound for the first term follows analogously. 
	
	\medskip
	
	\noindent\textbf{Case 1.} $\e^{-1}(t_{i+1}-t_{i})\ge 1$. We then use the long-time scaling to conclude
	\begin{align*}
		& \e^{\frac13}\left|\bar\calH_{\e^{-1}(t_{i+1}-t_{i})}((a_i+u-a_{i+1})\e^{-2/3})-\bar\calH_{\e^{-1}(t_{i+1}-t_{i})}((a_i-a_{i+1})\e^{-2/3})\right|
		\\ & \hspace{2cm}=\frac{\bar\h_{\e^{-1}(t_{i+1}-t_{i})}\left(\tfrac{a_i+u-a_{i+1}}{(t_{i+1}-t_i)^{2/3}}\right)-\bar\h_{\e^{-1}(t_{i+1}-t_{i})}\left(\tfrac{a_i-a_{i+1}}{(t_{i+1}-t_i)^{2/3}}\right)}{(t_{i+1}-t_i)^{-1/3}}
	\end{align*}
	where $\bar{\h}_s(x):=\h_s(x)+\frac{x^2}2$. We now consider two cases depending on the value of $u$.
	
	\medskip
	
	\noindent\textbf{Case 1.1.} Suppose $u\in [0,(t_{i+1}-t_i)^{2/3}]$.
	By Proposition \ref{p:ltkpzeq} \ref{p:ltlf} with $\gamma \mapsto \frac{u}{(t_{i+1}-t_i)^{2/3}}$, and using the fact that $\sqrt{\gamma} \le u^{1/4}(t_{i+1}-t_i)^{-1/3}$, we see that  $\eqref{B}\le \Con\exp(-\frac1{\Con}s^{3/2})$ for some $\Con>0$ in this case.
	
	\medskip
	
	\noindent\textbf{Case 1.2.} For $u\in [(t_{i+1}-t_i)^{2/3},1]$, we rely on one-point tail bounds. Indeed applying union bound we have
	\begin{align*}
		\eqref{A} & \le \Pr\left(\left|\frac{\bar\h_{\e^{-1}(t_{i+1}-t_{i})}\left(\tfrac{a_i+u-a_{i+1}}{(t_{i+1}-t_i)^{2/3}}\right)}{(t_{i+1}-t_i)^{-1/3}} \right|\ge\tfrac{s}8u^{1/4}\right)+\Pr\left(\left|\frac{\bar\h_{\e^{-1}(t_{i+1}-t_{i})}\left(\tfrac{a_i-a_{i+1}}{(t_{i+1}-t_i)^{2/3}}\right)}{(t_{i+1}-t_i)^{-1/3}}  \right|\ge\tfrac{s}8u^{1/4}\right)\\ & \le \Con\exp\left(-\tfrac1\Con s^{3/2}u^{3/8}(t_{i+1}-t_i)^{-1/2}\right) \le  \Con\exp\left(-\tfrac1{\Con}s^{3/2}\right).
	\end{align*}
	The penultimate inequality above follows from Proposition \ref{p:ltkpzeq} \ref{p:lstat}, \ref{p:ltail} and the last one follows from the fact $u\ge (t_{i+1}-t_i)^{2/3}$ and $t_{i+1}-t_i\in (0,1]$.
	
	\medskip
	
	\noindent\textbf{Case 2.} $\e^{-1}(t_{i+1}-t_{i})\le 1$. We here use the short-time scaling to conclude
	\begin{align*}
		& \e^{\frac13}\left|\bar\calH_{\e^{-1}(t_{i+1}-t_{i})}((a_i+u-a_{i+1})\e^{-2/3})-\bar\calH_{\e^{-1}(t_{i+1}-t_{i})}((a_i-a_{i+1})\e^{-2/3})\right|
		\\ & \hspace{1cm}=\big(\tfrac{\pi \e^{1/3}(t_{i+1}-t_i)}{4}\big)^{\frac14}\left[{\bar\g_{\e^{-1}(t_{i+1}-t_{i})}\left(\tfrac{2(a_i+u-a_{i+1})}{\sqrt{\pi \e^{1/3}(t_{i+1}-t_i)}}\right)-\bar\g_{\e^{-1}(t_{i+1}-t_{i})}\left(\tfrac{2(a_i-a_{i+1})}{\sqrt{\pi \e^{1/3}(t_{i+1}-t_i)}}\right)}\right]
	\end{align*}
	where $\bar{\g}_s(x):=\g_s(x)+\frac{(\pi s/4)^{3/4}x^2}{2s}$. We again consider two cases depending on the value of $u$.	 
	
	\medskip
	
	\noindent\textbf{Case 2.1.} Suppose $u\in (0,\frac{\sqrt{\pi}}{2}\e^{-1/3}(t_{i+1}-t_i))$. Then $\frac{2u}{\sqrt{\pi \e^{1/3}(t_{i+1}-t_i)}} < \sqrt{\e^{-1}(t_{i+1}-t_i)}$. This allows us to apply Proposition  \ref{p:stlf} with $\gamma\mapsto \frac{2u}{\sqrt{\pi \e^{1/3}(t_{i+1}-t_i)}}$ and $t\mapsto \e^{-1}(t_{i+1}-t_i)$. Using the fact that $u^{1/2} \le u^{1/4}$ for $u\in [0,1]$, we see that  $\eqref{B}\le \Con\exp(-\frac1{\Con}s^{3/2})$ for some $\Con>0$ in this case.
	
	\medskip
	
	\noindent\textbf{Case 2.2.} For $u\in [\frac{\sqrt{\pi}}{2}\e^{-1/3}(t_{i+1}-t_i),1]$, we rely on stationarity and one-point tail bounds (Proposition \ref{p:stkpzeq} \ref{p:sstat}, \ref{p:stail}). Indeed applying union bound we have
	\begin{align*}
		\eqref{B} & \le \Pr\left(\big(\tfrac{\pi \e^{1/3}(t_{i+1}-t_i)}{4}\big)^{\frac14}\left|\bar\g_{\e^{-1}(t_{i+1}-t_{i})}\left(\tfrac{2(a_i+u-a_{i+1})}{\sqrt{\pi \e^{1/3}(t_{i+1}-t_i)}}\right) \right|\ge\tfrac{s}8u^{1/4}\right) \\ & \hspace{3cm}+\Pr\left(\big(\tfrac{\pi \e^{1/3}(t_{i+1}-t_i)}{4}\big)^{\frac14}\left|\bar\g_{\e^{-1}(t_{i+1}-t_{i})}\left(\tfrac{2(a_i-a_{i+1})}{\sqrt{\pi \e^{1/3}(t_{i+1}-t_i)}}\right)  \right|\ge\tfrac{s}8u^{1/4}\right)\\ & \le \Con\exp\left(-\tfrac1\Con \left[s u^{1/4}(t_{i+1}-t_i)^{-1/4}\e^{-\frac1{12}}\right]^{3/2}\right).
	\end{align*}
	As $u\ge \frac{\sqrt{\pi}}{2}\e^{-1/3}(t_{i+1}-t_i)$, and $\e\in (0,1)$ we have $u^{1/4}(t_{i+1}-t_i)^{-1/4}\e^{-1/12} \ge \frac{\sqrt{\pi}}{2}$. Thus the last expression above is at most $\Con\exp\left(-\tfrac1{\Con}s^{3/2}\right)$.
	
	\medskip	 
	
	Combining the above two cases we have $\eqref{B} \le \Con\exp(-\frac1\Con s^{3/2})$ uniformly for $u\in [0,1]$. By the same argument one can show the term in \eqref{A} is also upper bounded by $\Con\exp(-\frac1\Con s^{3/2})$. This yields \eqref{eq:diff} for $i=1,2,\ldots,k-1$.
	
	Finally for $i=k$, observe that 
	\begin{align*}
		& \bar{F}(\vec{a} + e_k\cdot u) - \bar{F}(\vec{a}) \\ & = \e^{1/3}\left[\calH(a_{k-1}\e^{-2/3},\e^{-1}t_{k-1};(a_k+u)\e^{-2/3},\e^{-1})-\calH(a_{k-1}\e^{-2/3},\e^{-1}t_{k-1};a_k\e^{-2/3},\e^{-1})\right]  \\ & \hspace{4cm}+\frac{(a_{k-1}-a_k-u)^2-(a_{k-1}-a_k)^2}{2(1-t_{k-1})}.
	\end{align*}
	Then \eqref{eq:diff} follows for $i=k$ by the exact same computations as above. This completes the proof of the lemma.
\end{proof}

\subsection{Proof of Proposition \ref{p:ltd} and \ref{p:ltd.ptl}} \label{sec:twopt}
We now present the proofs of Proposition \ref{p:ltd} and \ref{p:ltd.ptl}. 

\begin{proof}[Proof of Proposition \ref{p:ltd}] We assume $m\ge 16\tau^2+1$. Otherwise the constant $\Con_1$ can be chosen large enough so that the inequality holds trivially. Without loss of generality assume $s<t$. We first consider the case when $s,t\in (0,1)$. Note that
	\begin{align*}
		& \Pr^{\xi}(|L_s^{(\e)}-L_t^{(\e)}|\ge m|s-t|^{\frac12-\delta}) \\  =& \iint\limits_{|u-v|\ge  m\e^{-2/3}|s-t|^{\frac12-\delta} } \frac{\calZ(0,0; u,\e^{-1}s)\calZ(u,\e^{-1}s;v, \e^{-1}t)\calZ(v,\e^{-1} t;x,\e^{-1})}{\calZ(0,0;x,\e^{-1})}\, \d u\, \d v.
	\end{align*}
	We make a change of variable $u=p\e^{-2/3}$,  $v=q\e^{-2/3}$ and $x= z \e^{-2/3}$. Then
	\begin{align}\notag
		& \Pr^{\xi}(|L_s^{(\e)}-L_t^{(\e)}|\ge m|s-t|^{\frac12-\delta})  \\= & \e^{-4/3} \iint\limits_{|p-q|\ge  m |s-t|^{\frac12-\delta} } \frac{\calZ(0,0; p\e^{-2/3},\frac{s}{\e})\calZ(p\e^{-2/3},\frac{s}{\e};q\e^{-2/3}, \frac{t}{\e})\calZ(q\e^{-2/3},\frac{t}{\e};z\e^{-2/3},\frac1{\e})}{\calZ(0,0;z\e^{-2/3},\e^{-1})}\, \d q\, \d p .\label{e:int3}
	\end{align}
	Recall the multivariate spatial process $F_{\vec{t};\e}(\vec{x})$ from \eqref{def:f}. Take $k=3$ and set $\vec{t}=(s,t,1)$, and $\vec{x}=(p,q,z)$. We also set $$B(\vec{t}\,):=\ind\{s\le \e\}\log\sqrt{2\pi\frac{s}{\e}}+\ind\{t-s\le \e\}\log\sqrt{2\pi\frac{t-s}{\e}}+\ind\{1-t\le \e\}\log\sqrt{2\pi\frac{1-t}{\e}}.$$
	For the numerator of the integrand in \eqref{e:int3}  observe that
	\begin{align}\label{e:iden}
		& \calZ(0,0; p\e^{-\frac23},\tfrac{s}{\e})\calZ(p\e^{-\frac23},\tfrac{s}{\e};q\e^{-\frac23}, \tfrac{t}{\e})\calZ(q\e^{-\frac23},\tfrac{t}{\e};z\e^{-\frac23},\tfrac1{\e})
		=
		\exp\left[\e^{-\frac13}F_{\vec{t};\e}(\vec{x})-\tfrac{\e^{-1}}{24}-B(\vec{t}\,)\right].
	\end{align}
	Set $M=\frac{m^2}{64}$. Applying Lemma \ref{l:ltail} with $\nu=\frac12$ and $s=M$, we see that with probability greater than $1- \Con \exp(-\frac{1}{\Con}M^{3/2})$, 
	\begin{align}\label{eqs}
		\mbox{r.h.s.~of \eqref{e:iden}} \le \exp\left[\e^{-1/3}M-\e^{-1/3}\left(\tfrac{p^2}{4s}+\tfrac{(q-p)^2}{4(t-s)}+\tfrac{(z-q)^2}{4(1-t)}\right)-\tfrac{\e^{-1}}{24}-B(\vec{t}\,)\right].
	\end{align}
	On the other hand, for the denominator of the integrand in \eqref{e:int3} by one-point long-time tail bound from Proposition \ref{p:ltkpzeq} with probability at least $1-\Con\exp(-\frac1{\Con}M^{3/2})$ we have
	\begin{align*}
		\calZ(0,0;z\e^{-2/3},\e^{-1}) \ge \exp\left(\e^{-1/3}\h_{\e^{-1}}(z)-\tfrac{\e^{-1}}{24}\right) \ge \exp\left(-\e^{-1/3}(M+\tfrac12\tau^2)-\tfrac{\e^{-1}}{24}\right).
	\end{align*}
	Combining the previous equation with \eqref{eqs} we get that with probability at least $1-\Con\exp(-\frac1{\Con}M^{3/2})$ we have
	\begin{align} \notag
		\mbox{r.h.s.~of \eqref{e:int3}} & \le \e^{-\frac43} \exp\left(\e^{-1/3}(2M+\tfrac12\tau^2)-B(\vec{t}\,)\right)\cdot \\ & \hspace{2cm}\iint_{|p-q|\ge m|s-t|^{\frac12-\delta}} \exp\left[-\e^{-1/3}\left(\tfrac{p^2}{4s}+\tfrac{(q-p)^2}{4(t-s)}+\tfrac{(z-q)^2}{4(1-t)}\right)\right]dq\, dp \notag \\ & \le \e^{-\frac43} \exp\left(\e^{-\frac13}(2M+\tfrac12\tau^2-\tfrac{m^2}{4|t-s|^{2\delta}})-B(\vec{t}\,)\right)\iint_{\R^2} \exp\left[-\e^{-\frac13}\left(\tfrac{p^2}{4s}+\tfrac{r^2}{4(1-t)}\right)\right]dr\, dp \notag \\ & = 4\pi\sqrt{s(1-t)}\e^{-1} \exp\left(\e^{-\frac13}(2M+\tfrac12\tau^2-\tfrac{m^2}{4|t-s|^{2\delta}})-B(\vec{t}\,)\right). \label{e:la}
	\end{align}
	Observe that 
	\begin{align}
		\label{use}
		{\sqrt{r}}\exp\left(-\ind\{r\le \e\}\log\sqrt{\tfrac{2\pi r}{\e}}\right) \le 1.
	\end{align}
	As $M=\frac{m^2}{64}$ we have $2M-\frac{m^2}{4|t-s|^{2\delta}} \le -\frac{m^2}{8|t-s|^{2\delta}}$. Furthermore $\frac12\tau^2 \le \frac{m^2}{16|t-s|^{2\delta}}$ under the condition $m\ge 16\tau^2+1$. Thus,
	\begin{align*}
		\mbox{r.h.s.~of \eqref{e:la}} \le 4\pi \e^{-1} \exp\left(-\e^{-\frac13}\tfrac{m^2}{16|t-s|^{2\delta}}-\ind\{t-s\le \e\}\log\sqrt{\tfrac{2\pi(t-s)}{\e}}\right).
	\end{align*}
	Clearly the last expression is at most $\Con_1\exp(-\frac1{\Con_1}m^2)$ for some $\Con_1>0$ depending on $\tau,\delta$. This bound holds uniformly over $t,s\in (0,1)$ with $t\neq s$ and $\e\in (0,1)$. This concludes the proof for $s,t\in (0,1)$.
	
	{Finally, when $s=0$ we have 
		\begin{align*}
		& \Pr^{\xi}(|L_t^{(\e)}|\ge m|t|^{\frac12-\delta})  = \iint\limits_{|v|\ge  m\e^{-2/3}|t|^{\frac12-\delta} } \frac{\calZ(0,;v, \e^{-1}t)\calZ(v,\e^{-1} t;x,\e^{-1})}{\calZ(0,0;x,\e^{-1})}\,  \d v.
	\end{align*}
	The proof can now be completed by following the argument for $s,t\in (0,1)$ case. Indeed, the denominator can be bounded by the exact same manner as above, whereas the numerator can be controlled with the $k=2$  version of Lemma \ref{l:ltail}. The case $t=1$ is analogous to the case $s=0$. We have thus established Proposition \ref{p:ltd}}.
\end{proof}
\begin{proof}[Proof of Proposition \ref{p:ltd.ptl}] We now explain how the above proof can be modified to extend it to the point-to-line version. Fix any $m>0$ and $M>1$. Indeed observe that for $0<s<t<1$, one has 
	\begin{align}\notag
		& \Pr_*^{\xi}(|L_{s,*}^{(\e)}-L_{t,*}^{(\e)}|\ge m|s-t|^{\frac12-\delta})  \\= & \e^{-\frac43} \int_{\R}\iint\limits_{|p-q|\ge  m |s-t|^{\frac12-\delta} } \frac{\calZ(0,0; p\e^{-\frac23},\frac{s}{\e})\calZ(p\e^{-\frac23},\frac{s}{\e};q\e^{-\frac23}, \frac{t}{\e})\calZ(q\e^{-\frac23},\frac{t}{\e};z\e^{-\frac23},\frac1{\e})}{\int_{\R}\calZ(0,0;y\e^{-\frac23},\e^{-1})\d y}\, \d q\, \d p \, \d z.  \label{ers}
	\end{align}
	Since Lemma \ref{l:ltail} is a process-level estimate that allows even the endpoint to vary, \eqref{eqs} continues to hold simultaneously for all $p,q,z\in \R$ with same high probability.  However for the lower bound on the denominator, one-point lower-tail bound is not sufficient. Instead, for the denominator we use long-time process-level lower bound from Proposition 4.1 in \cite{cgh} to get that with probability at least $1-\Con\exp(-\frac1\Con M^{3/2})$ we have
	\begin{align*}
		\int_{\R} \calZ(0,0;y\e^{-2/3},\e^{-1})\d y \ge \int_{\R}\exp\left(-\tfrac{M+y^2}{\e^{1/3}}-\tfrac{\e^{-1}}{24}\right)\d y \ge \Con \e^{\frac16}\exp\left(-\e^{-1/3}M-\tfrac{\e^{-1}}{24}\right). 
	\end{align*}
	{Combining the
previous equation with \eqref{eqs}} we get that with probability at least $1-\Con\exp(-\frac1{\Con}M^{3/2})$ we have
	\begin{equation}\label{ers2}
	    \begin{aligned} 
		\mbox{r.h.s.~of \eqref{ers}} & \le \e^{-\frac32} \exp\left(2\e^{-1/3}M-B(\vec{t}\,)\right)\cdot \\ & \hspace{2cm}\int_{\R}\iint_{|p-q|\ge m|s-t|^{\frac12-\delta}} \exp\left[-\e^{-1/3}\left(\tfrac{p^2}{4s}+\tfrac{(q-p)^2}{4(t-s)}+\tfrac{(z-q)^2}{4(1-t)}\right)\right]dq\, dp\, dz. 
	\end{aligned}
	\end{equation}
	{On $|p-q|\ge m|s-t|^{\frac12-\delta}$, we have $(q-p)^2/4(t-s) \ge (q-p)^2/8(t-s)+m^2/8|t-s|^{2\delta}$. Applying this inequality followed by expanding the range of integration we get}
	\begin{align} \notag
		\mbox{r.h.s.~of \eqref{ers2}} & \le  \e^{-\frac32} \exp\left(\e^{-\frac13}(2M-\tfrac{m^2}{8|t-s|^{2\delta}})-B(\vec{t}\,)\right) \notag \\ & \hspace{2cm}\cdot\int_{\R}\int_{\R}\int_{\R} \exp\left[-\e^{-\frac13}\left(\tfrac{p^2}{4s}+\tfrac{r^2}{8(t-s)}+\tfrac{u^2}{4(1-t)}\right)\right]dq\, dr\, du \notag \\ & = \sqrt{2^7\pi^3s(1-t)(t-s)} \cdot \e^{-1} \exp\left(\e^{-\frac13}(2M-\tfrac{m^2}{8|t-s|^{2\delta}})-B(\vec{t}\,)\right). \notag
	\end{align}
	Just as in the proof of Proposition \ref{p:ltd}, setting $M=\frac{m^2}{64}$, and using \eqref{use}, the above expression can be shown to be at most $\Con\exp(-\frac1\Con m^2)$ uniformly over $\e\in (0,1)$ and $0<s<t<1$. This establishes the proposition.
\end{proof}

\section{Annealed Convergence for short-time and long-time}\label{sec:annconv}

In this section we prove our main results. In Section \ref{sec4.1} we prove Theorems \ref{thm:ann_short}, \ref{ltight}, and \ref{ltight.ptl}. In Section \ref{sec4.2}, we show Theorem \ref{thm:ann_long_pr} assuming Conjecture \ref{conj:sheet}. 

\subsection{Proof of Theorems \ref{thm:ann_short}, \ref{ltight}, and \ref{ltight.ptl}} \label{sec4.1}

In this section we prove results related to short-time and long-time tightness and related pointwise weak convergence. While the proof of long-time tightness relies on modulus of continuity estimates from Proposition \ref{p:ltd} and Proposition \ref{p:ltd.ptl},  the proof of short-time tightness relies on the following Brownian relation of annealed law of $\cdrp$.

\begin{lemma}[Brownian Relation] \label{l3} Let $X\sim \cdrp(0,0;0,t)$ and $Y\sim \cdrp(0,0;*,t)$.  For any continuous functional $\mathcal{L} :C([0,t])\to \R$ we have
	\begin{align}\label{br}
		\Ex\left[\calZ(0,0;0,t)\sqrt{2\pi t}\cdot\mathcal{L}(X)\right]=\Ex(\mathcal{L}(B)), \quad \Ex\left[\calZ(0,0;*,t)\cdot\mathcal{L}(Y)\right]=\Ex(\mathcal{L}(B_*))
	\end{align}
	where $B_*$ and $B$ are standard Brownian motion and standard Brownian bridge on $[0,t]$ respectively. 
\end{lemma}
{\begin{remark}  Note that $\sqrt{2\pi t}=\frac1{p(0,t)}$ where $p(0,t)$ is the heat kernel. Since the Brownian bridge finite-dimensional densities are product of heat kernels divided by $p(0,t)$, this additional factor $\frac1{p(0,t)}$ is required in the point-to-point version for appropriate comparison to the the Brownian bridge law (see \eqref{bbd} below).
\end{remark}}
\begin{proof} Take $0=t_0<t_1<\cdots<t_{k}<t_{k+1}=t$. The Brownian motion identity appears as Lemma 4.2 in \cite{akq}. To show the bridge version note that by Definition \ref{def:cdrp}, the quantity $$\calZ(0,0;0,t)\Pr^{\xi}(X({t_1})\in \d x_1,\ldots,X({t_k})\in \d x_k)$$ is product of independent random variables with mean $p(x_{j+1}-x_j,t_{j+1}-t_{j})$ where $p(x,t)$ denotes the heat kernel. Noting that $p(0,t)=\frac1{\sqrt{2\pi t}}$, and recalling the finite-dimensional distribution of Brownian bridge (Problem 6.11 in \cite{ks}) we get that
\begin{equation}\label{bbd}
 \begin{aligned}
		\Ex\left[\calZ(0,0;0,t)\sqrt{2\pi t} \cdot\Pr^{\xi}(X({t_1})\in \d x_1,\ldots,X({t_k})\in \d x_k)\right] & =\frac1{p(0,t)}\prod_{j=0}^{k}p(x_{j+1}-x_j,t_{j+1}-t_{j}) \\ & =\Pr\left(B(t_1)\in \d x_1, \ldots, B(t_k) \in \d x_k\right).
	\end{aligned}   
\end{equation}
	\eqref{br} now follows from the above by approximation of $\mathcal{L}$ with
	simple functions.
\end{proof}

\begin{proof}[Proof of Theorem \ref{thm:ann_short}] We first show finite-dimensional convergence. Fix $0=t_0<t_1<\cdots<t_{k+1}=1$. Take $x_1,\ldots,x_k\in \R$. Set $x_0=0$ and $x_{k+1}=0$. Note that the density for $(Y_{t_i}^{(\e)})_{i=1}^k$ at $(x_i)_{i=1}^k$ is given by
	\begin{align*}
		f_{\vec{t};\e}(\vec{x}):= \frac{\e^{k/2} }{\calZ(0,0;0,\e)}\prod_{j=0}^{k}\calZ(\sqrt{\e} x_j,\e t_j;\sqrt{\e} x_{j+1},\e t_{j+1}).
	\end{align*}
	For a Brownian bridge $B$ on $[0,1]$ starting at $0$ and ending at $x$, the density for $(B_{t_i})_{i=1}^k$ at $(x_i)_{i=1}^k$ is given by
	\begin{align*}
		g_{\vec{t}}(\vec{x}):= \frac{1 }{p(0,1)}\prod_{j=0}^{k}p(x_{j+1}- x_j, t_{j+1}-t_j) 
	\end{align*}
	where $p(x,t)=\frac1{\sqrt{2\pi t}}e^{-x^2/2t}$. Using the distributional identities for $\calZ$ (see Remark \ref{re:gscal}) and using Equation (8.11) in \cite{comets} and Brownian scaling, we deduce 
	\begin{align*}
		\frac{\calZ(\sqrt{\e}x_j,\e t_j;\sqrt{\e}x_{j+1},\e t_{j+1})}{p(\sqrt{\e}(x_{j+1}-x_j),\e(t_{j+1}-t_j))}\stackrel{d}{=}\Ex_{0,0}^{t_{j+1}-t_j,x_{j+1}-x_j}\left[:\exp:\left\{\e^{1/4}\int_0^{t_{j+1}-t_j}{\xi}(s,B(s))ds\right\}\right]
	\end{align*}
	where $B$ is a Brownian bridge conditioned $B(0)=0$ and $B(t_{j+1}-t_j)=x_{j+1}-x_j$. The expectation above is taken w.r.t.~this Brownian bridge only. Here $:\exp:$ denotes the Wick exponential (see \cite{comets} for details). The right side of the above equation is a random variable (function of the noise $\xi$). We claim that this random variable converges to $1$ in probability. Indeed using chaos expansion, and Lemma 2.4 in \cite{exact}, it follows that for every fixed $t,x$ we have
	\begin{align*}
		\Ex\left[\left\{\Ex_{0,0}^{t,x}\left[:\exp:\left\{\e^{1/4}\int_0^{t}{\xi}(s,B(s))ds\right\}\right]-1\right\}^2\right]=\sqrt{\e}\sum_{k=1}^{\infty}\frac{\e^{(k-1)/2}t^{k/2}}{(4\pi)^{k/2}}\frac{(\Gamma(1/2)^k)}{\Gamma(k/2)}.
	\end{align*}
	The above sum converges. Thus as $\e\downarrow 0$, the above expression goes to zero, proving the claim. As $p(\sqrt\e x,\e t)=\e^{-1/2}p(x,t)$, we thus have $f_{\vec{t};\e}(\vec{x}) \stackrel{p}{\to} g_{\vec{t}}(\vec{x})$. Thus the quenched finite-dimensional density of $Y^{(\e)}$ converges in probability to the finite-dimensional density of the Brownian Bridge. We now show that the same holds for the annealed law. Indeed, note that $\left|g_{\vec{t}}(\vec{x})-f_{\vec{t};\e}(\vec{x})  \right|^{+}$ converges to zero in probability and is bounded above by $g_{\vec{t}}(\vec{x})$. Thus by DCT and Jensen's inequality, we obtain
	\begin{align*}
		\left|g_{\vec{t}}(\vec{x})-\Ex [f_{\vec{t};\e}(\vec{x})]\right|^{+} \le \Ex_{\xi}\left|g_{\vec{t}}(\vec{x})-f_{\vec{t};\e}(\vec{x})\right|^{+} \to 0
	\end{align*}
	as $\e \downarrow 0$. Now by Scheffe's theorem, it follows that the annealed finite-dimensional distribution of $Y^{(\e)}$ converges weakly to the finite-dimensional distribution of the Brownian bridge.
	
	Let us now verify tightness. Recall that $X(\e t) =\sqrt{\e} Y_t^{(\e)}$. Observe that by union bound followed by Markov inequality we have 
	\begin{align*}
		\Pr\left[\sup_{\substack{0\le t,s\le 1\\ |t-s|\le \delta}}|Y_t^{(\e)}-Y_s^{(\e)}|\ge \eta\right] & \le \Pr\left[\calZ(0,0;0,\e)\sqrt{2\pi\e}\sup_{\substack{0\le t,s\le 1\\ |t-s|\le \delta}}|Y_t^{(\e)}-Y_s^{(\e)}|\ge {\eta}\delta^{1/3}\right] \\ & \hspace{4cm}+\Pr\left[\calZ(0,0;0,\e)\sqrt{2\pi\e}\le \delta^{1/3}\right]\\ & \le \frac{\sqrt{2\pi}}{\eta\delta^{1/3}}\Ex\left[\calZ(0,0;0,\e)\sup_{\substack{0\le t,s\le 1\\ |t-s|\le \delta}}|X({\e t})-X({\e s})|\right] \\ & \hspace{4cm}+\Pr\left[\g_{\e}(0)\le (4\e/\pi)^{-1/4}\log (\delta^{1/3})\right].
	\end{align*}
	Note that by one-point short-time tail bounds from Proposition \ref{p:stkpzeq} \ref{p:stail}, the second expression above goes to zero as $\delta\downarrow 0$ uniformly in $\e\le 1$. For the first expression,  by Lemma \ref{l3} we have
	\begin{align*}
		\Ex\left[\calZ(0,0;0,\e)\sup_{\substack{0\le t,s\le 1\\ |t-s|\le \delta}}|X({\e t})-X({\e s})|\right]=\frac1{\sqrt{2\pi\e}}\Ex\left[\sup_{\substack{0\le t,s\le 1\\ |t-s|\le \delta}}|B_{\e t}'-B_{\e s}'|\right],
	\end{align*}
	where $B'$ is a Brownian bridge on $[0,\e]$. By scaling property of Brownian bridges we may write the last expression simply as $$\tfrac1{\sqrt{2\pi}}\Ex\left[\sup_{\substack{0\le t,s\le 1\\ |t-s|\le \delta}}|B_{t}-B_{s}-(t-s)B_1|\right]$$ where $B$ is a Brownian motion on $[0,1]$. This expression is free of $\e$ and by \cite[Lemma 1]{fis} this goes to zero with rate $O(\delta^{1/2-\gamma})$ for any $\gamma>0$.  Thus we have shown
	\begin{align*}
		\limsup_{\delta \downarrow 0}\sup_{\e\in (0,1)}\Pr\left[\sup_{\substack{0\le t,s\le 1\\ |t-s|\le \delta}}|Y_t^{(\e)}-Y_s^{(\e)}|\ge \eta\right] =0.
	\end{align*}
	Since $Y_{0}^{(\e)}=0$, by standard criterion of tightness (see Theorem 4.10 in \cite{ks}) combined with finite-dimensional convergence shown before, we have weak convergence to Brownian Bridge. This completes the proof.
\end{proof}

\begin{proof}[Proof of Theorem \ref{ltight}] Let us first prove \ref{tght} using Corollary \ref{p:amc}. Fix $\gamma\in (0,1)$. {We consider $\beta\in (0,1)$ small enough so that $\gamma \ge \rho(\beta)$ where $\rho(\beta):=\sup_{t\in (0,\beta]} t^{\frac14}\log\frac2t$}. Taking $\delta=\frac14$,  the estimates in \eqref{e:lamc} ensure that for all $\e\in (0,1)$ we have
	\begin{align*}
		\Pr\left(\sup_{t\neq s\in [0,1], |t-s|<\beta} |L_s^{(\e)}-L_t^{(\e)}|\ge \gamma\right) & \le \Pr\left(\sup_{t\neq s\in [0,1], |t-s|<\beta} \frac{|L_s^{(\e)}-L_t^{(\e)}|}{|t-s|^{\frac14}\log\frac{2}{|t-s|}}\ge \frac{\gamma}{\rho(\beta)}\right) \\ & \le\Con \exp\left(-\tfrac1{\Con}\tfrac{\gamma^2}{\rho(\beta)^2}\right).
	\end{align*}
	Note that as $\beta \downarrow 0$, we have $\rho(\beta)\downarrow 0$. Hence
	\begin{align*}
		\limsup_{\beta\downarrow 0}\sup_{\e\in (0,1)}\Pr\left(\sup_{t\neq s\in [0,1], |t-s|<\beta} |L_s^{(\e)}-L_t^{(\e)}|\ge \gamma\right)=0.
	\end{align*}
	Since $L_0^{(\e)}=0$, the above modulus of continuity estimate yields tightness for the process $L_t^{(\e)}$.
	
	\medskip
	
	\noindent For \ref{pt}, let us fix $t\in (0,1)$ and consider $V\sim \cdrp(0,0;0,\e^{-1})$. Let $\mathcal{M}_{t,\e^{-1}}$ denote the unique mode of the quenched density of $V(\e^{-1}t)$. By \cite[Theorem 1.4]{dz22}, we know $\mathcal{M}_{t,\e^{-1}}$ exists uniquely {almost surely}. By \cite[Corollary 7.3]{dz22} we have
	\begin{align*}
		\limsup_{K\to\infty}\limsup_{\e\downarrow 0} \Pr^{\xi}(|V(\e^{-1}t)-\mathcal{M}_{t,\e^{-1}}|\ge K)=0, \mbox{ in probability}.
	\end{align*}
	Applying reverse Fatou's Lemma we have
	\begin{align*}
		\limsup_{K\to\infty}\limsup_{\e\downarrow 0} \Pr(|V(\e^{-1}t)-\mathcal{M}_{t,\e^{-1}}|\ge K)=0.
	\end{align*}
	Thus in particular, $\e^{-\frac23}[V(\e^{-1}t)-\mathcal{M}_{t,\e^{-1}}] \stackrel{p}{\to} 0$. However,  $\e^{-2/3}\mathcal{M}_{t,\e^{-1}}\stackrel{d}{\to} \Gamma(t\sqrt{2})$ due to \cite[Theorem 1.8]{dz22}. This proves \ref{pt}.
\end{proof}

\begin{proof}[Proof of Theorem \ref{ltight.ptl}] Let us first prove part \ref{stght} {which claims short-time process convergence. We first show finite-dimensional convergence. Fix $0=t_0<t_1<\cdots<t_{k+1}=1$. Take $x_1,\ldots,x_k\in \R$. Set $x_0=0$ and $x_{k+1}=*$. Note that the density for $(Y_{*}^{(\e)}(t_i))_{i=1}^k$ at $(x_i)_{i=1}^k$ is given by
	\begin{align*}
		f_{\vec{t};\e}^*(\vec{x}):= \frac{\e^{k/2} }{\calZ(0,0;*,\e)}\prod_{j=0}^{k}\calZ(\sqrt{\e} x_j,\e t_j;\sqrt{\e} x_{j+1},\e t_{j+1}).
	\end{align*}
From the finite-dimensional convergence argument in proof of Theorem \ref{thm:ann_short} we know that
\begin{align} \label{tos}
    \e^{k/2}\prod_{j=0}^{k-1}\calZ(\sqrt{\e} x_j,\e t_j;\sqrt{\e} x_{j+1},\e t_{j+1}) \stackrel{p}{\to} \prod_{j=0}^{k-1} p(x_{j+1}-x_j,t_{j+1}-t_j)=:g_{\vec{t}}^*(\vec{x}).
\end{align}
Note that $g_{\vec{t}}^*(\vec{x})$ is the finite-dimensional density for the standard Brownian motion. We now claim that
\begin{align}\label{tosh}
    \mathcal{Z}(0,0;*,\e) \stackrel{p}{\to} 1, \quad \mathcal{Z}(\sqrt{\e}x_{k-1},\e t_{k-1}; *,{\e} t_k) \stackrel{p}{\to} 1.
\end{align}
Combining \eqref{tos} and \eqref{tosh} we have that 
	 $f_{\vec{t};\e}^*(\vec{x}) \stackrel{p}{\to} g_{\vec{t}}^*(\vec{x})$ which implies quenched finite-dimensional density convergence. This convergence can then be upgraded to annealed finite-dimensional density convergence by the same argument of the proof of Theorem \ref{thm:ann_short}.}
	 
	 {We thus focus on proving \eqref{tosh}. To prove the first part of \eqref{tosh} we utilize the short-time scaling from \eqref{eq:kpzscal} to get
\begin{align}\label{tosh1}
    \calZ(0,0;*,\e) = \int_{\R} e^{\calH(x,t)}dx = \frac1{\sqrt{2\pi \e}} \int_{\R} \exp\left((\tfrac{\pi \e}{4})^{1/4}\g_{\e}\big(x\sqrt{\tfrac{4}{\pi \e}}\big)\right)dx.
\end{align}
Fix any $\nu\in (0,1)$. Applying \cite[Proposition 4.4]{dg} (with $s=\e^{-\frac16}$) we get that with probability at least $1-\Con \exp(-\frac1\Con \e^{-\frac14})$
\begin{align}
    -\frac{(\pi \e/4)^{3/4}(1+\nu)x^2}{2\e}-\e^{-1/6} \le \g_{\e}\big(x\big) \le -\frac{(\pi \e/4)^{3/4}(1-\nu)x^2}{2\e}+\e^{-1/6}, \ \ \mbox{for all }x\in \R,
\end{align}
where the constant $\Con$ depends on $\nu$.
Inserting the above inequality in \eqref{tosh1} we get that with probability at least $1-\Con \exp(-\frac1\Con \e^{-\frac14})$
$$\exp\left(-(\tfrac{\pi}4)^{1/4}\e^{1/12}\right) \frac1{\sqrt{2\pi \e}} \int_{\R} e^{-\frac{(1+\nu)x^2}{2\e}}dx \le \mathcal{Z}(0,0;*,\e) \le  \exp\left((\tfrac{\pi}4)^{1/4}\e^{1/12}\right) \frac1{\sqrt{2\pi \e}} \int_{\R} e^{-\frac{(1-\nu)x^2}{2\e}}dx.$$
Thus
$$\Pr\left(\frac{\exp\left(-(\tfrac{\pi}4)^{1/4}\e^{1/12}\right)}{\sqrt{1+\nu}} \le \mathcal{Z}(0,0;*,\e) \le   \frac{\exp\left((\tfrac{\pi}4)^{1/4}\e^{1/12}\right)}{\sqrt{1-\nu}}\right) \ge 1-\Con \exp(-\tfrac1\Con \e^{-\frac14}),$$
which implies
$$\limsup_{\e\to \infty}\Pr\left(\tfrac{1}{\sqrt{1+\nu}} \le \mathcal{Z}(0,0;*,\e) \le   \tfrac{1}{\sqrt{1-\nu}}\right)=1.$$
Taking $\nu\downarrow 0$, we get the first part of \eqref{tosh}. The second part follows analogously.} 
	 
	 {Let us now verify tightness.} Observe that by union bound followed by Markov inequality we have 
	\begin{equation}\label{sht}
		\begin{aligned}
			\Pr\left[\sup_{\substack{0\le t,s\le 1\\ |t-s|\le \delta}}|Y_{*}^{(\e)}(t)-Y_{*}^{(\e)}(t)|\ge \eta\right] & \le \Pr\left[\calZ(0,0;*,\e)\sup_{\substack{0\le t,s\le 1\\ |t-s|\le \delta}}|Y_{*}^{(\e)}(t)-Y_{*}^{(\e)}(s)|\ge {\eta}\delta^{1/3}\right] \\ & \hspace{4cm}+\Pr\left[\calZ(0,0;*,\e)\le \delta^{1/3}\right]\\ & \le \frac{1}{\eta\delta^{1/3}}\Ex\left[\calZ(0,0;*,\e)\frac1{\sqrt\e}\sup_{\substack{0\le t,s\le 1\\ |t-s|\le \delta}}|X({\e t})-X({\e s})|\right] \\ & \hspace{4cm}+\Pr\left[\g_{\e}(*)\le \e^{-1/4}\log (\delta^{1/3})\right],
		\end{aligned}
	\end{equation}
	where
	\begin{align*}
		\g_{\e}(*) & :=\e^{-1/4}\log\calZ(0,0;*,{\e})  \\ & =\e^{-1/4}\left[-\log\sqrt{2\pi \e}+\log \int_{\R} \exp\left((\tfrac{\pi \e}4)^{1/4}\g_{\e}(\sqrt{\tfrac{4}{\pi \e}}x)\right) \d x\right]
	\end{align*}
	with $\g_{\e}(x)$ defined in \eqref{eq:kpzscal}. Let us now bound each term in the r.h.s.~of \eqref{sht} separately. For the second term we claim that
	\begin{align}\label{gstar}
		\limsup_{\delta\downarrow 0}\sup_{\e\in (0,1)}\Pr\left[\g_{\e}(*)\le \e^{-1/4}\log (\delta^{1/3})\right] =0.
	\end{align}
	Note that by Proposition 4.4 in \cite{dg} (the infimum process bound with $\nu=1$) we have for any $s>0$ with probability at least $1-\Con\exp(-\frac1\Con s^{3/2})$,
	\begin{align*}
		(\tfrac{\pi \e}4)^{1/4}\g_{\e}(\sqrt{\tfrac{4}{\pi \e}}x) & \ge -(\tfrac{\pi \e}4)^{1/4}\left[s+(\tfrac{\pi \e}{4})^{3/4}\cdot \tfrac{4}{\pi \e} \tfrac{x^2}{\e}\right] = -(\tfrac{\pi \e}4)^{1/4}s-\tfrac{x^2}{\e}, \mbox{ for all } x\in \R.
	\end{align*}
	Thus, with probability at least $1-\Con\exp(-\frac1\Con s^{3/2})$, \begin{align*}
		\g_{\e}(*) & \ge \e^{-1/4}\left[-\log\sqrt{2\pi \e}+\log\left(\int_{\R} \exp\left(-(\tfrac{\pi \e}4)^{1/4}s-\tfrac{x^2}{\e}\right)\d x\right)\right] \\ & = \e^{-1/4}\left[-\log\sqrt{2\pi \e}+\log \left( \sqrt{\pi \e} \exp\left(-(\tfrac{\pi \e}4)^{1/4}s\right)\right)\right] \\ &  =\e^{-1/4}\left[-\log\sqrt{2}-(\tfrac{\pi \e}4)^{1/4}s\right] \ge -s-\e^{-1/4}\log 2.
	\end{align*}
	Now we take $s=-\e^{-1/4}\log(2\delta^{1/6})$ which is positive for $\delta$ small enough. Then $-s-\e^{-1/4}\log 2 =\frac12\e^{-1/4}\log(\delta^{1/3})>\e^{-1/4}\log(\delta^{1/3})$. Hence uniformly in all $\e\in (0,1)$, with probability at least $1-\Con\exp(-\frac1{\Con}[-\log(2\delta^{1/6})]^{3/2})$, we have $\g_{\e}(*)\ge \e^{-1/4}\log(\delta^{1/3})$. This verifies \eqref{gstar}. 
	
	Next for the first expression on r.h.s.~of \eqref{sht}, by  Lemma \ref{l3} we have
	\begin{align*}
		\Ex\left[\calZ(0,0;*,\e)\frac1{\sqrt\e}\sup_{\substack{0\le t,s\le 1\\ |t-s|\le \delta}}|X({\e t})-X({\e s})|\right]=\frac1{\sqrt\e}\Ex\left[\sup_{\substack{0\le t,s\le 1\\ |t-s|\le \delta}}|B_{\e t}'-B_{\e s}'|\right],
	\end{align*}
	where $B'$ is a Brownian motion on $[0,\e]$. By scaling property of Brownian motion we may write the last expression simply as $$\Ex\left[\sup_{\substack{0\le t,s\le 1\\ |t-s|\le \delta}}|B_{t}-B_{s}|\right]$$ where $B$ is a Brownian motion on $[0,1]$. This expression is free of $\e$ and by \cite[Lemma 1]{fis} this goes to zero with rate $O(\delta^{1/2-\gamma})$ for any $\gamma>0$.  Thus we have shown
	\begin{align*}
		\limsup_{\delta \downarrow 0}\sup_{\e\in (0,1)}\Pr\left[\sup_{\substack{0\le t,s\le 1\\ |t-s|\le \delta}}|Y_*^{(\e)}(t)-Y_*^{(\e)}(s)|\ge \eta\right] =0.
	\end{align*}
	Since $Y_{*}^{(\e)}(0)=0$, {this proves tightness. Along with finite-dimensional convergence, this establishes part \ref{stght}}.
	
	\medskip
	
	The tightness results in part \ref{ltght} follows via the same arguments as in the proof of Theorem \ref{ltight} \ref{tght} utilizing the point-to-line modulus of continuity from Proposition \ref{p:qmc.ptl}. For part \ref{ptl}, we rely on localization results from \cite{dz22}. Indeed, by Theorem 1.5 in \cite{dz22}, we know the quenched density of $V(\e^{-1})$ (recall $V\sim \cdrp(0,0;*,\e^{-1})$) has a unique mode $\mathcal{M}_{*,\e^{-1}}$ almost surely. By the same argument as in the proof of Theorem \ref{ltight} \ref{pt}, the point-to-line version of Corollary 7.3 in \cite{dz22} leads to the fact that $\e^{-2/3}[L_{*}^{(\e)}(1)-\mathcal{M}_{*,\e^{-1}}] \stackrel{p}{\to} 0$. Finally from Theorem 1.8 in \cite{dz22} we have $\e^{-2/3}\mathcal{M}_{*,\e^{-1}} \stackrel{d}{\to} 2^{1/3}\mathcal{M}$. This establishes \ref{ptl}.
\end{proof}

\subsection{Proof of Theorem \ref{thm:ann_long_pr} modulo Conjecture \ref{conj:sheet}} \label{sec4.2}

In this section we prove Theorem \ref{thm:ann_long_pr} assuming Conjecture \ref{conj:sheet}. The proof also relies on a technical result which we first state below.

\begin{lemma}[Deterministic convergence]\label{l5}
	Let $f(x): \R^k \rightarrow \R$ be a continuous function with a unique maximizer $\vec{a}\in \R^k$ and $f_{\e}(x): \R^k \rightarrow \R$ be a sequence of continuous functions that converges to $f(x)$ uniformly over compact subsets. Fix any $\delta>0$ and take $M>0$ so that $(a_i-\delta,a_i+\delta)\in [-M,M]$ for all $i$. For $x\in \R$, set $$g_{\e}(x) := \frac{\exp(\e^{-\frac{1}{3}}f_{\e}(\vec{x}))}{\int_{[-M,M]^k} \exp(\e^{-\frac{1}{3}}f_{\e}(\vec{y}))\d \vec{y}}.$$ For all $\vec{b} \in [-M,M]^k$, we have:
	
	\begin{equation}\label{limsup}
		\limsup_{\e \downarrow 0}\int_{-M}^{b_1} \cdots \int_{-M}^{b_k} g_{\e}(\vec{x})\d \vec{x}  \le \prod_{i=1}^k\ind\{ a_i\le b_i + \delta\},
	\end{equation}
	\begin{equation}\label{liminf}
		\liminf_{\e \downarrow 0}\int_{-M}^{b_1} \cdots \int_{-M}^{b_k} g_{\e}(\vec{x})\d \vec{x}  \ge \prod_{i=1}^k\ind\{ a_i\le b_i -\delta\}.
	\end{equation}
\end{lemma}

Proof of this lemma follows via standard real analysis and hence we defer its proof to the end of this section. We now proceed to prove Theorem \ref{thm:ann_long_pr} assuming the above lemma.

\begin{proof}[Proof of Theorem \ref{thm:ann_long_pr}] For clarity we split the proof into three steps.
	
	\medskip
	
	\noindent\textbf{Step 1.} Fix $0= t_0 < t_1 < \ldots < t_k < t_{k+1} = 1$. For convenience set $\Gamma_{t_i}:=\Gamma(t_i\sqrt{2})$ where $\Gamma(\cdot)$ is the geodesic of directed landscape from $(0,0)$ to $(0,\sqrt2)$. Consider any $\vec{a} = (a_1, \ldots, a_k) \in \R^k,$ which is a continuity point for the CDF of $(\Gamma_{t_i})_{i=1}^k$. For any $M\ge \sup_{i} |a_i|+1$, define \begin{align}
		\label{va}
		V_{\vec{a}}(M):=[-M,a_i]\times \cdots \times [-M,a_k]\subset \R^k.
	\end{align}
	To show convergence in finite-dimensional distribution, it suffices to prove that as $\e\downarrow 0$
	\begin{align}\label{fddcv}
		\Pr\left(\bigcap_{i=1}^k \{L_{t_i}^{(\e)}\le a_i\}\right)  \to \Pr\left(\bigcap_{i=1}^k \{\Gamma_{t_i}\le a_i\}\right).
	\end{align}
	From Definition \ref{def:cdrp} and using the long-time scaling from \eqref{eq:kpzscal}, we obtain that the joint density of $(L_{t_1}^{(\e)}, L_{t_2}^{(\e)}, \ldots,L_{t_k}^{(\e)}) $ at $(x_i)_{i=1}^k$ is given by $$\frac{g_{\vec{t};\e}(\vec{x})}{\int_{\R^k} g_{\vec{t};\e}(\vec{y}) d\vec{y}}, \qquad g_{\vec{t};\e}(\vec{x}):=\exp(\e^{-1/3} U_{\vec{t};\e}(\vec{x}))$$
	where
	\begin{align}\label{ftex}
	    U_{\vec{t};\e}(\vec{x}) : = \sum_{i=1}^{k+1}(t_i - t_{i-1})^{1/3}\h_{\e^{-1}t_{i-1}, \e^{-1}t_{i}}((t_i - t_{i-1})^{-2/3}x_{i-1}, (t_i - t_{i-1})^{-2/3} x_{i})
	\end{align}
	Here $x_0 =x_{k+1} = 1$.

	In this step, we reduce our computation to understanding the integral behavior of $g_{\vec{t};\e}$ on a compact set. More precisely, the goal of this step is to show there there exists a constant $\Con>0$ such that for all $M$ large enough
	\begin{align} \label{step1}
		\left|\Pr\left(\bigcap_{i=1}^k \{L_{t_i}^{(\e)}\le a_i\}\right)-\Ex\left[\frac{\int_{V_{\vec{a}}(M)} g_{\vec{t};\e}(\vec{y})d\vec{y}}{\int_{[-M,M]^k} g_{\vec{t};\e}(\vec{y})d\vec{y}} \right]\right|\le \Con\exp\left(-\tfrac1\Con M^2\right)
	\end{align}
	where $V_{\vec{a}}(M)$ is defined in \eqref{va}. We proceed to prove \eqref{step1} by demonstrating appropriate lower and upper bounds. For upper bound observe that by union bound we have
	\begin{align} \notag
		\Pr\left(\bigcap_{i=1}^k \{L_{t_i}^{(\e)}\le a_i\}\right) & \le \Pr\left(\bigcap_{i=1}^k \{L_{t_i}^{(\e)}\in [-M,a_i]\}\right)+\Pr\left(\sup_{t\in [0,1]} |L_{t}^{(\e)}|\ge M\right) \\ & \le \Ex\left[\frac{\int_{V_{\vec{a}}(M)} g_{\vec{t};\e}(\vec{y})d\vec{y}}{\int_{[-M,M]^k} g_{\vec{t};\e}(\vec{y})d\vec{y}} \right]+\Pr\left(\sup_{t\in [0,1]} |L_{t}^{(\e)}|\ge M\right) \notag
		\\ & \le \Ex\left[\frac{\int_{V_{\vec{a}}(M)} g_{\vec{t};\e}(\vec{y})d\vec{y}}{\int_{[-M,M]^k} g_{\vec{t};\e}(\vec{y})d\vec{y}} \right]+\Con \exp\left(-\tfrac1\Con M^2\right) \label{up1}
	\end{align}
	where the last inequality follows from Corollary \ref{l:lqmcc1} for some constant $\Con>0$. For the lower bound we have
	\begin{align} \notag
		\Pr\left(\bigcap_{i=1}^k \{L_{t_i}^{(\e)}\le a_i\}\right) & \ge \Pr\left(\bigcap_{i=1}^k \{L_{t_i}^{(\e)}\in [-M,a_i]\}\right) \\ & = \Ex\left[\frac{\int_{V_{\vec{a}}(M)} g_{\vec{t};\e}(\vec{y})d\vec{y}}{\int_{[-M,M]^k} g_{\vec{t};\e}(\vec{y})d\vec{y}}\cdot \frac{\int_{[-M,M]^k} g_{\vec{t};\e}(\vec{y})d\vec{y}}{\int_{\R^k} g_{\vec{t};\e}(\vec{y})d\vec{y}} \right] \notag \\ & \ge \Ex\left[\frac{\int_{V_{\vec{a}}(M)} g_{\vec{t};\e}(\vec{y})d\vec{y}}{\int_{[-M,M]^k} g_{\vec{t};\e}(\vec{y})d\vec{y}}\cdot \Pr^{\xi}\left(\sup_{t\in [0,1]}|L_t^{(\e)}|\le M\right) \right]. \label{e1}
	\end{align}
	By Corollary \ref{l:lqmcc1} we see that there exist two constants $\Con_1,\Con_2>0$ such that with probability at least $1-\Con_2\exp(-\frac1{\Con_2}M^3)$, the random variable $\Pr^{\xi}\left(\sup_{t\in [0,1]}|L_t^{(\e)}|\le M\right)$ is at least $1-\Con_1\exp(-\frac1{\Con_1}M^2)$. Thus,
	\begin{align} \notag
		\mbox{r.h.s.~of \eqref{e1}} & \ge \left[1-\Con_2\exp\left(-\tfrac1{\Con_2}M^3\right)\right]\Ex\left[\frac{\int_{V_{\vec{a}}(M)} g_{\vec{t};\e}(\vec{y})d\vec{y}}{\int_{[-M,M]^k} g_{\vec{t};\e}(\vec{y})d\vec{y}}\cdot \left[1-\Con_1\exp\left(-\tfrac1{\Con_1}M^2\right)\right]\right] \\ & \ge \Ex\left[\frac{\int_{V_{\vec{a}}(M)} g_{\vec{t};\e}(\vec{y})d\vec{y}}{\int_{[-M,M]^k} g_{\vec{t};\e}(\vec{y})d\vec{y}}\right]-\Con_1\exp\left(-\tfrac1{\Con_1}M^2\right). \label{lw1}
	\end{align}
	In view of \eqref{up1} and \eqref{lw1}, we thus arrive at \eqref{step1} by adjusting the constants. This completes our work for this step.
	
	\medskip
	
	\noindent\textbf{Step 2.} In this step, we discuss how directed landscape and hence the geodesic appear in the limit. {Recall the random function $U_{\vec{t};\varepsilon}(\vec{x})$ from \eqref{ftex}}. We exploit Conjecture \ref{conj:sheet}, to show that as $\e\downarrow 0$, as $\R^k$-valued processes we have the following convergence in law 
	\begin{align}\label{conv}
		U_{\vec{t};\e}(\vec{x}) \stackrel{d}{\to} \mathbf{U}_{\vec{t}}(\vec{x}):=2^{-\frac13}\sum_{i=1}^{k+1}\mathcal{L}(x_{i-1},t_{i-1}\sqrt{2};x_{i},t_{i}\sqrt{2})
	\end{align}
	in the uniform-on-compact topology. Here $\mathcal{L}(x,s;y,t)$ denotes the directed landscape. Note that by Definition \ref{def:geo}, $(\Gamma_{t_i})_{i=1}^k$ is precisely the almost sure unique $k$-point maximizer of $\mathbf{f}_{\vec{t}}(\vec{x})$.
	
	To show \eqref{conv}, we rely on Conjecture \ref{conj:sheet} heavily. Indeed, assuming Conjecture \ref{conj:sheet}, for each $i$, as $\e\downarrow 0$  we have
	\begin{align*}
		& \h_{\e^{-1}t_{i-1}, \e^{-1}t_{i}}((t_i - t_{i-1})^{-2/3}x, (t_i - t_{i-1})^{-2/3} y) \\ & \hspace{2cm}\stackrel{d}{\to} 2^{-1/3}\mathcal{S}^{(i)}\big(2^{-1/3}(t_i - t_{i-1})^{-2/3}x,2^{-1/3}(t_i - t_{i-1})^{-2/3}y\big)
	\end{align*}
	where the convergence holds under the uniform-on-compact topology. Here $\mathcal{S}^{(i)}$ are independent Airy sheets as $\h_{\e^{-1}t_{i-1},\e^{-1}t_i}(\cdot,\cdot)$ are independent. Now by the definition of directed landscape we have
	\begin{align*}
		& 2^{-\frac13}\sum_{i=1}^{k+1}\mathcal{L}(x_{i-1},t_{i-1}\sqrt{2};x_{i},t_{i}\sqrt{2}) \\ & \hspace{2cm} \stackrel{d}{=} 2^{-\frac13}\sum_{i=1}^{k+1}(t_{i+1}-t_i)^{1/3}\mathcal{S}^{(i)}\big(2^{-1/3}(t_i - t_{i-1})^{-2/3}x_{i-1},2^{-1/3}(t_i - t_{i-1})^{-2/3}x_i\big)
	\end{align*}
	with $x_0=x_{k+1}=1$. Here the equality in distribution holds as $\R^k$-valued processes in $\vec{x}$.  This allow us to conclude the desired convergence for $U_{\vec{t};\e}(\vec{x})$ in \eqref{conv}, completing our work for this step.
	
	\medskip
	
	\noindent\textbf{Step 3.} In this step, we complete the proof of \eqref{fddcv} utilizing \eqref{step1} and the weak convergence in \eqref{conv}.  Using Skorokhod's representation theorem, given any fixed $M$, we may assume that we are working on a probability space where
	\begin{align*}
		\Pr(\m{A})=1, \quad \mbox{for } \ \m{A}: = \bigg\{\sup_{\vec{x}\in [-M,M]^k} \left|U_{\vec{t};\e}(\vec{x})-\mathbf{U}_{\vec{t}}(\vec{x})\right| \to 0\bigg\}.
	\end{align*}
	Let us define
	\begin{align*}
		(\Gamma_{t_i}(M))_{i=1}^k:=\underset{\vec{x}\in [-M,M]^k}{\operatorname{argmax}} \mathbf{f}_{\vec{t}}(\vec{x}),
	\end{align*}
	where in case there are multiple maximizers we take the one whose sum of coordinates is the largest. We next define
	\begin{align*}
		\m{B}: = \bigg\{\underset{\vec{x}\in [-M,M]^k}{\operatorname{argmax}} \mathbf{U}_{\vec{t}}(\vec{x}) \mbox{ exists uniquely and } (\Gamma_{t_i}(M))_{i=1}^k \in [-\tfrac{M}2,\tfrac{M}2]^k\bigg\}.
	\end{align*} 
	Fix any $\delta\in (0,\frac{M}{2})$. By Lemma \ref{l5} we have
	\begin{equation} \label{up2}
		\begin{aligned}
			\limsup_{\e\downarrow 0}\Ex\left[\frac{\int_{V_{\vec{a}}(M)} g_{\vec{t};\e}(\vec{y})d\vec{y}}{\int_{[-M,M]^k} g_{\vec{t};\e}(\vec{y})d\vec{y}} \right] & \le \Pr(\neg\m{B})+\Ex\left[\limsup_{\e\downarrow 0}\frac{\int_{V_{\vec{a}}(M)} g_{\vec{t};\e}(\vec{y})d\vec{y}}{\int_{[-M,M]^k} g_{\vec{t};\e}(\vec{y})d\vec{y}}\ind\{\m{A}\cap\m{B}\}\right] \\ & \le \Pr(\neg \m{B})+\Pr\left(\bigcap_{i=1}^k \{\Gamma_{t_i}(M) \le a_i+\delta\}\right)  \\ & \le \Pr(\neg \m{B})+\Pr\left(\bigcap_{i=1}^k \{\Gamma_{t_i} \le a_i+\delta\}\right)+\Pr\left(\sup_{t\in [0,1]} |\Gamma_{t}|\ge M\right),
		\end{aligned}
	\end{equation}
	where the last inequality follows by observing that $\Gamma_{t_i}(M)=\Gamma_{t_i}$ for all $i$, whenever $\sup_{t\in [0,1]} |\Gamma_{t}| \le M$ (and the fact that $\Gamma(\cdot)$ exists uniquely almost surely via Theorem 12.1 in \cite{dov}). In the same manner we have
	\begin{equation}\label{lw2}
		\begin{aligned}
			\liminf_{\e\downarrow 0}\Ex\left[\frac{\int_{V_{\vec{a}}(M)} g_{\vec{t};\e}(\vec{y})d\vec{y}}{\int_{[-M,M]^k} g_{\vec{t};\e}(\vec{y})d\vec{y}} \right] & \ge \Ex\left[\liminf_{\e\downarrow 0}\frac{\int_{V_{\vec{a}}(M)} g_{\vec{t};\e}(\vec{y})d\vec{y}}{\int_{[-M,M]^k} g_{\vec{t};\e}(\vec{y})d\vec{y}}\ind\{\m{A}\cap\m{B}\}\right] \\ & \ge \Pr\left(\bigcap_{i=1}^k \{\Gamma_{t_i}(M) \le a_i-\delta\}, \m{A}\cap\m{B}\right)  \\ & \ge \Pr\left(\bigcap_{i=1}^k \{\Gamma_{t_i} \le a_i-\delta\}\right)-\Pr(\neg\m{B})-\Pr\left(\sup_{t\in [0,1]} |\Gamma_{t}|\ge M\right).
		\end{aligned}
	\end{equation}
	By Proposition 12.3 in \cite{dov}, $$\Pr(\neg \m{B}) \le \Pr\left(\sup_{t\in [0,1]} |\Gamma_{t}|\ge M\right)\le \Con\exp\left(-\tfrac1\Con M^3\right).$$
	
	Thus taking $M\uparrow \infty$, followed by $\delta\downarrow 0$, and using the fact that $\vec{a}$ is a continuity point of the density on both sides of \eqref{up2} and \eqref{lw2} we have
	\begin{align*}
		\lim_{M\to \infty}\limsup_{\e\downarrow 0} \Ex\left[\frac{\int_{V_{\vec{a}}(M)} g_{\vec{t};\e}(\vec{y})d\vec{y}}{\int_{[-M,M]^k} g_{\vec{t};\e}(\vec{y})d\vec{y}} \right] = \lim_{M\to \infty}\liminf_{\e\downarrow 0} \Ex\left[\frac{\int_{V_{\vec{a}}(M)} g_{\vec{t};\e}(\vec{y})d\vec{y}}{\int_{[-M,M]^k} g_{\vec{t};\e}(\vec{y})d\vec{y}} \right]= \Pr\left(\bigcap_{i=1}^k \{\Gamma_{t_i} \le a_i\}\right)
	\end{align*}
	Combining this with \eqref{step1} we thus arrive at \eqref{fddcv}. This completes the proof.
\end{proof}
\begin{proof}[Proof of Lemma \ref{l5}]
	We begin by proving \eqref{limsup}. When $a_i \le b_i + \delta$ for all $i$, the r.h.s of \eqref{limsup} is 1 whereas the l.h.s of \eqref{limsup} is always less than 1. Thus we focus on when $a_j>b_j+\delta$ for some $j$. In that case $\vec{a} \notin [-M,b_1]\times \cdots \times [-M,b_k]$. As $\vec{a}$ is the unique maximizer of the continuous function $f(\vec{x})$, there exists $\eta>0$ such that
	$$\sup_{y_i \in [-M, b_i], i=1,2,\ldots,k}f(\vec{y}) < f(\vec{a})-\eta.$$
	By uniform convergence over compacts, we can get $\e_0$ such that 
	\begin{align*}
		\sup_{\e \le \e_0}\sup_{\vec{x} \in [-M, M]^k}|f_{\e}(\vec{x})-f(\vec{x})| < \tfrac14\eta.
	\end{align*}
	By continuity of $f$ at $\vec{a}$, we can get $\delta_0<\delta$ such that for all $0\le \rho\le \delta$ we have
	\begin{align*}
		\sup_{x_i \in [a_i-\rho,a_i+\rho], i=1,\ldots,k}|f(\vec{x})-f(\vec{a})| < \tfrac14\eta.
	\end{align*}
	Thus for all $\e\le \e_0$ and $0\le \rho\le \delta_0$ we have $f_{\e}(\vec{x})\ge f(\vec{a})-\tfrac12\eta$ for all $\vec{x}$ with $x_i\in [a_i-\rho,a_i+\rho]$. And for all $\e\le \e_0$, $f_{\e}(\vec{y})<f(\vec{a})-\tfrac34\eta$ for all $\vec{y}$ with $y_i\in [-M,b_i]$. Thus in conclusion
	\begin{align*}
		\int_{-M}^{b_1} \cdots \int_{-M}^{b_k} \exp(\e^{-\frac13}f_{\e}(\vec{x}))\d \vec{x} \le (2M)^k\exp(\e^{-1/3}[f(\vec{a})-\tfrac34\eta)])
	\end{align*}
	and
	\begin{align*}
		\int_{[-M,M]^k} \exp(\e^{-\frac13}f_{\e}(\vec{x}))\d \vec{x} \ge \int_{a_1-\delta_0}^{a_1+\delta_0}\cdots \int_{a_k-\delta_0}^{a_k+\delta_0} \exp(\e^{-\frac13}f_{\e}(\vec{x}))\d \vec{x} \ge  {(2\delta_0)^k\exp(\e^{-1/3}[f(\vec{a})-\tfrac12\eta)])}.
	\end{align*}
	Combining the above two bounds we have
	\begin{align*}
		\int_{-M}^{b_1} \cdots \int_{-M}^{b_k} g_{\e}(\vec{x})\d \vec{x} \le (\tfrac{M}{\delta_0})^k\exp(-\tfrac14\e^{-1/3}\eta),
	\end{align*}
	which goes to zero as $\e\downarrow 0$.  Thus, we conclude the proof of \eqref{limsup}. 
	The proof of \eqref{liminf} follows analogously.
\end{proof}

\appendix

\section{Proof of Lemma \ref{l:mon}} \label{app1}
In this section, we prove Lemma \ref{l:mon}. The idea is to view short-time scaled KPZ equation $\g_t(\cdot)$ defined in \eqref{eq:kpzscal} as the lowest index curve of an appropriate line ensemble and use certain stochastic monotonicity properties of the same. To make our exposition self-contained, below we briefly introduce the line ensemble machinery. 

\smallskip

Fix $t>0$ throughout this section and consider the convex function $$\mathbf{G}_t(x)=(\pi t/4)^{1/2}e^{(\pi t/4)^{1/4} x}.$$ Recall the general notion of line ensembles from Section 2 in \cite{CH14}. {Let $\mathcal{L}=(\mathcal{L}_1,\mathcal{L}_2,\ldots)$ be an $\mathbb{N}\times \mathbb{R}$ indexed line ensemble. Fix $k_1\le k_2$ with $k_1,k_2\in \N$ and an interval $(a,b)\in \R$ and two vectors $\vec{x},\vec{y} \in \R^{k_2-k_1+1}$. {Let $\mathbb{P}^{k_1, k_2, (a,b), \vec{x}, \vec{y}}_{\mathrm{free}}$ denote the law of $k_2-k_1+1$ many independent Brownian bridges taking values $\vec{x}$ at time $a$ and $\vec{y}$ at time $b$. Given two measurable functions $f,g:(a,b)\to \R\cup \{\pm\infty\}$, the law $\mathbb{P}^{k_1, k_2, (a,b), \vec{x}, \vec{y}, f,g}_{\mathbf{G}_t}$ on $\mathcal{L}_{k_1},\ldots,\mathcal{L}_{k_2} : (a,b) \to \R$ has the following Radon-Nikodym derivative w.r.t.~ $\mathbb{P}^{k_1, k_2, (a,b), \vec{x}, \vec{y}}_{\mathrm{free}}$}: 
	\begin{align}\label{am1}
		\frac{d\mathbb{P}^{k_1, k_2,(a,b), \vec{x}, \vec{y}, f, g}_{\mathbf{G}_t}}{d\mathbb{P}^{k_1, k_2, (a,b), \vec{x}, \vec{y}}_{\mathrm{free}}}(\mathcal{L}_{k_1}, \ldots , \mathcal{L}_{k_2}) &= \frac{\exp\bigg\{- \sum_{i=k_1}^{k_2+1} \int \mathbf{G}_t\big(\mathcal{L}_{i}(x)- \mathcal{L}_{i-1}(x)\big) dx \bigg\} }{Z^{k_1, k_2, (a,b), \vec{x}, \vec{y}, f,g}_{\mathbf{G}_t}},
	\end{align}
	where  $\mathcal{L}_{k_1-1}=f$, or    $\infty$ if $k_1=1$; and  $\mathcal{L}_{k_2+1}=g$. Here $Z^{k_1, k_2, (a,b), \vec{x}, \vec{y}, f, g}_{\mathbf{G}_t}$ is the normalizing constant which produces a probability measure.} {We say $\mathcal{L}$ enjoys the \emph{$\mathbf{G}_t$-Brownian Gibbs property} if, for all $K = \{k_1,\ldots, k_2\}\subset \mathbb{N}$ and $(a,b)\subset \R$, the following distributional equality holds:
	\begin{align} \label{a0}
		\mathrm{Law}\Big(\mathcal{L}_{K\times (a,b)} \text{ conditioned on }\mathcal{L}_{\mathbb{N}\times \R \backslash K \times (a,b)}\Big) = \mathbb{P}^{k_1,k_2, (a,b),\vec{x}, \vec{y}, f, g}_{\mathbf{G}_t} \, ,
	\end{align}
	where $\vec{x}= (\mathcal{L}_{k_1}(a), \ldots , \mathcal{L}_{k_2}(a))$, $\vec{y} =(\mathcal{L}_{k_1}(b),\ldots , \mathcal{L}_{k_2}(b))$, and where again $\mathcal{L}_{k_1-1}=f$, or    $\infty$ if $k_1=1$; and  $\mathcal{L}_{k_2+1}=g$. }

Similar to the Markov property, a \emph{strong} version of the $\mathbf{G}_t$-Brownian Gibbs property that is valid with respect to \textit{stopping domains} exists. A pair $(\mathfrak{a},\mathfrak{b})$ of random variables is called a $K$-stopping domain if $\big\{\mathfrak{a} \leq a, \mathfrak{b}\geq b\big\} \in \mathfrak{F}_{\textrm{ext}}\big(K\times (a,b)\big)$, the $\sigma$-field generated by $\mathcal{L}_{(\N\times \R)\setminus (K\times (a,b))}$. $\mathcal{L}$ satisfies the strong $\mathbf{G}_t$-Brownian Gibbs property if for all $K = \{k_1,\ldots, k_2\}\subset \mathbb{N}$ and $K$-stopping domain if $(\mathfrak{a}, \mathfrak{b})$, the conditional distribution of $\mathcal{L}_{K\times (\mathfrak{a}, \mathfrak{b})}$ given $\mathfrak{F}_{\textrm{ext}}\big(K\times (\mathfrak{a},\mathfrak{b})\big)$ is $\mathbb{P}^{k_1,k_2, (l,r),\vec{x}, \vec{y}, f, g}_{\mathbf{G}_t}$, where $\ell = \mathfrak{a}$, $r= \mathfrak{b}$, $\vec{x} =(\mathcal{L}_i(\mathfrak{a}))_{i\in K}$, $\vec{y} =(\mathcal{L}_i(\mathfrak{b}))_{i\in K}$, and where again $\mathcal{L}_{k_1-1}=f$, or    $\infty$ if $k_1=1$; and  $\mathcal{L}_{k_2+1}=g$. 

\smallskip

The following lemma shows how the short-time scaled KPZ process $\g_t(\cdot)$ fits into a line ensemble satisfying the $\mathbf{G}_{t}$-Brownian Gibbs property. 
\begin{lemma}[Lemma 2.5 in \cite{dg} and Lemma 2.5 of \cite{CH16}] \label{line-ensemble}
	For each $t>0$, there exists an $\N\times \R$-indexed line ensemble $\{\g^{(n)}_{t}(x)\}_{n\in \N, x\in \R}$ satisfying the $\mathbf{G}_{t}$-Brownian Gibbs property and the lowest indexed curve $\g^{(1)}_{t}(x)$ is equal in distribution (as a process in $x$) to  $\g_t(x)$ defined in \eqref{eq:kpzscal}. Furthermore, the line ensemble  $\{\g^{(n)}_{t}(x)\}_{n\in \N, x\in \R}$ satisfies the strong $\mathbf{G}_t$-Brownian Gibbs property.
\end{lemma}

Before beginning the proof of Lemma \ref{l:mon} we recall one more property of line ensembles, i.e. the stochastic monotonicity, which is indispensable to the study of monotone events in Lemma \ref{l:mon}.

\begin{lemma}[Lemmas 2.6 and 2.7 of \cite{CH16}]\label{Coupling1}
	Fix a finite interval $(a,b)\subset \R$ and $x,y\in \R$. For $i\in \{1,2\}$, fix measurable functions  $g_i: (a,b) \to \R \cup \{-\infty\}$ such that  $g_2(s)\leq  g_1(s)$ for $s\in (a,b)$. For each $v \in \{1,2\}$, let
	$\Pr_v$ denote the law $\mathbb{P}^{1, 1, (a,b), x,y,+\infty,g_v}_{\mathbf{G}_t}$, so that a $\Pr_v$-distributed random variable
	$\mathcal{R}_i= \{\mathcal{R}_v(s)\}_{s\in(a,b)}$ is a random function on $[a,b]$ with endpoints $x$ and $y$. Then a common probability space may be constructed on which the two measures are supported such that, almost surely, $\mathcal{R}_{1}(s)\geq \mathcal{R}_{2}(s)$ for all $s\in (a,b)$.
\end{lemma}

\begin{proof}[Proof of Lemma \ref{l:mon}] Fix an interval $[a,b]$ and a corresponding monotone set $A\in \mathcal{B}(C([a,b]))$. By Lemma \ref{line-ensemble} and tower property of expectation we may write
	\begin{align}
		\notag \Pr\left[\g_t(\cdot)\mid_{[a,b]} \ \in A \mid (\g_t(x))_{x\not\in (a,b)}\right] & = \Ex^{(\ge 2)}\left[\Pr\left[\g_t^{(1)}(\cdot)\mid_{[a,b]} \ \in A \mid (\g_t^{(n)}(\cdot))_{n\ge 2}, (\g_t^{(1)}(x))_{x\not\in (a,b)}\right]\right] \\ & =  \Ex^{(\ge 2)}\left[\mathbb{P}^{1,1, (a,b),\g_t^{(1)}(a),\g_t^{(1)}(b),+\infty, \g_t^{(2)}(\cdot)}_{\mathbf{G}_t}\left(\g_t^{(1)}(\cdot)\mid_{[a,b]} \ \in A\right)\right]\label{a1}
	\end{align}
	where the last equality follows from \eqref{a0}. Here $\Ex^{(\ge 2)}$ denotes the expectation operator taken over all lower curves $\{\g_t^{(n)}(\cdot)\}_{n\ge 2}$. Now by Lemma \ref{Coupling1}, decreasing $\g_t^{(2)}(\cdot)$ pointwise on $[a,b]$ reduces the value of $\g_t^{(1)}(\cdot)$ pointwise stochastically. But by the definition of monotone set $A$ (see \eqref{tom}), we know decreasing $\g_t^{(1)}(\cdot)\mid_{[a,b]}$ stochastically pointwise and keeping the endpoint fixed, only increases the conditional probability appearing above. Thus, we may drop $\g_t^{(2)}(\cdot)$ all the way to $-\infty$, to obtain 
	\begin{align}\label{a2}
		\mbox{r.h.s.~of \eqref{a1}} \le \Ex^{(\ge 2)}\left[\mathbb{P}^{1,1, (a,b),\g_t^{(1)}(a),\g_t^{(1)}(b),+\infty, -\infty}_{\mathbf{G}_t}\left(\g_t^{(1)}(\cdot)\mid_{[a,b]} \ \in A\right)\right].
	\end{align}
	Under the above situation the Radon-Nikodym derivative appearing in \eqref{am1} becomes constant, and thus
	\begin{align*}
		\mathbb{P}^{1,1, (a,b),\g_t^{(1)}(a),\g_t^{(1)}(b),+\infty, -\infty}_{\mathbf{G}_t}\left[\cdot\right]=\mathbb{P}_{\operatorname{free}}^{1,1, (a,b),\g_t^{(1)}(a),\g_t^{(1)}(b)}\left[\cdot\right].
	\end{align*}
	The measure on the right side above is a single Brownian bridge measure on $[a,b]$ starting at $\g_t^{(1)}(a)$ and ending at $\g_t^{(1)}(b)$ and hence free of $\{\g_t^{(n)}(\cdot)\}_{n\ge 2}$. Thus r.h.s.~of \eqref{a2} can be viewed as $\Pr_{\operatorname{free}}^{(a,b),(\g_t^{(1)}(a),\g_t^{(1)}(b))}(A)$.  This establishes \eqref{mon}. The case when $[a,b]$ is a stopping domain follows from the same calculation and the fact that $\{\g_t^{(n)}(\cdot)\}_{n\ge 1}$ satisfies the strong $\mathbf{G}_t$-Brownian Gibbs property via Lemma \ref{line-ensemble}. 
\end{proof}

\bibliographystyle{alphaabbr}		
\bibliography{paths}
\end{document}